 \documentclass[11pt,reqno]{amsart}
 \usepackage[top=3cm,bottom=3cm,right=3.3cm,left=3.3cm]{geometry}

\usepackage{amsmath,amsthm}
\usepackage{txfonts} 
\usepackage{hyperref}

\usepackage{multicol}
\usepackage{graphicx}
\graphicspath{{./figures/}}
\newcommand{\term}[1]{{\textbf{#1}}}
\newcommand{\dupeterm}{\term}

\usepackage[shortlabels]{enumitem}

\usepackage[
  backend=biber,
  style=alphabetic,
  sorting=nyt,
  doi=false,
  isbn=false,
  backref=true,
  url=false,
  hyperref=true,
  maxbibnames=99,
  maxcitenames=5]{biblatex}
\DeclareFieldFormat*{date}{\textbf{#1}} 
\DeclareFieldFormat[article]{journaltitle}{\textit{\textbf{#1}}} 
\DeclareFieldFormat[book,incollection]{series}{\textit{\textbf{#1}}} 
\AtEveryBibitem{\clearlist{language}} 

\newbibmacro{preferred-url}[1]{%
  \iffieldundef{doi}{%
    \iffieldundef{url}{#1}{%
      \href{\thefield{url}}{#1}%
    }
  }{\href{http://dx.doi.org/\thefield{doi}}{#1}}
}
\DeclareFieldFormat*{title}{\usebibmacro{preferred-url}{\textit{#1}}}

\theoremstyle{plain}
\newtheorem{introtheorem}{Theorem}

\newtheorem{introprop}[introtheorem]{Proposition}
\theoremstyle{plain}
\newtheorem{theo}{Theorem}[section]
\newtheorem{prop}[theo]{Proposition}
\newtheorem{lemma}[theo]{Lemma}
\newtheorem{coro}[theo]{Corollary}
\newtheorem{claim}[theo]{Claim}
\theoremstyle{definition}
\newtheorem{defi}[theo]{Definition}
\newtheorem{ques}[theo]{Question}

\newtheorem{example}[theo]{Example}
\newtheorem{remark}[theo]{Remark}

\usepackage{pgf,accents} 
\makeatletter
\newcommand*{\pgfunderrightarrow}{%
  \@ifstar
    {\let\ifpgf@depth\iftrue\mathpalette\@pgfunderrightarrow}
    {\let\ifpgf@depth\iffalse\mathpalette\@pgfunderrightarrow}%
}
\newcommand*{\@pgfunderrightarrow}[2]{%
  #2%
  \edef\pgf@math@fam{\the\fam}%
  \pgfpicture
    \pgfsetbaseline{0pt}
    \pgf@relevantforpicturesizefalse      
    \pgfsetroundcap                       
    \pgfsetarrowsend{to}
    \pgfutil@tempdima=0.28pt%
    \advance\pgfutil@tempdima by.8\pgflinewidth%
    \pgfutil@tempdima-4\pgfutil@tempdima
    \sbox\pgfutil@tempboxa{$\m@th\fam\pgf@math@fam#1#2$}%
    \advance\pgfutil@tempdima-\dp\pgfutil@tempboxa
    \pgfutil@tempdimb\wd\pgfutil@tempboxa
    \pgfpathmoveto{-\pgfqpoint{-5pt}{\pgfutil@tempdima}}%
    \pgfpathlineto{\pgfqpoint{0\pgfutil@tempdimb}{\pgfutil@tempdima}}%
    \pgfusepath{stroke}
    \ifpgf@depth
      \pgf@relevantforpicturesizetrue
      \pgfpathmoveto{\pgfqpoint{0pt}{-\pgfutil@tempdimb}}%
      \pgfusepath{use as bounding box}%
    \fi
  \endpgfpicture
}
\makeatother

\newcommand{\wt}{\widetilde}
\newcommand{\ov}{\overline}

\newcommand{\eps}{\varepsilon}

\newcommand{\NN}{\mathbb{N}} 
\newcommand{\ZZ}{\mathbb{Z}} 
\newcommand{\QQ}{\mathbb{Q}} 
\newcommand{\RR}{\mathbb{R}} 
\newcommand{\CC}{\mathbb{C}} 

\DeclareMathOperator{\mmod}{mod} 
\DeclareMathOperator{\Conv}{Conv} 
\DeclareMathOperator{\ext}{ext} 

\newcommand{\Sph}{\mathbb S} 
\newcommand{\DD}{\mathbb D} 
\newcommand{\UU}{\mathbb{U}} 
\newcommand{\U}{\mathrm{T}^1} 
\renewcommand{\H}{\mathrm{H}} 
\newcommand{\HSgammaR}{\H_2({\U\Sigma}, {\vecgamma}; \RR)}
\newcommand{\HSgammaZ}{\H_2({\U\Sigma}, {\vecgamma}; \ZZ)}

\newcommand{\FF}{\mathcal{F}} 
\newcommand{\Fs}{\FF^s}       
\newcommand{\Fu}{\FF^u}       
\newcommand{\fgeod}{\varphi_{\mathrm{geod}}} 
\newcommand{\fgeodt}{(\fgeod^t)_{t\in\RR}}
\newcommand{\fxt}{(\varphi_X^t)_{t\in\RR}}
\newcommand{\ft}{(\varphi^t)_{t\in\RR}}
\newcommand{\Schw}[1]{\mathcal{S}chw_{#1}} 

\newcommand{\xTh}{{\bf{x}_M}}  
\DeclareMathOperator{\Len}{Len} 
\newcommand{\intnorm}{\bf x} 
\newcommand{\xx}{{\intnorm}_\gamma} 
\newcommand{\xxx}[1]{\xx(#1)}
\newcommand{\BBx}{B^*_{\xx}} 
\DeclareMathOperator{\Eul}{Eul} 

\newcommand{\BB}{S\!^{BB}} 
\newcommand{\Snu}{\BB(\eta)}
\newcommand{\Snnu}{\BB(-\eta)}
\newcommand{\Snuun}{\BB(\eta_1)}
\newcommand{\Snudeux}{\BB(\eta_2)}

\newcommand{\abas}{\pgfunderrightarrow{\alpha}}
\newcommand{\arev}{{\overset\leftarrow{\alpha}}}

\newcommand{\syma}{{\overset\leftrightarrow{\alpha}}}
\newcommand{\gbas}{\pgfunderrightarrow{\gamma}_i}

\newcommand{\vecgamma}{\accentset{\leftrightarrow}{\gamma}}
\newcommand{\symgamma}{{\overset\leftrightarrow{\gamma}}}

\newcommand{\OO}{\mathcal{O}}       
\newcommand{\Cnu}{{S^\times(\eta)}} 

\title[Intersection norms and Birkhoff cross sections]
{Intersection norms on surfaces\\
and Birkhoff sections for geodesic flows}
\author{Marcos Cossarini}
\author{Pierre Dehornoy}
\date{first version: April 2016; this version: November 2024}

\begin{document}
\maketitle
\vspace{-.5cm}

\begin{abstract}
Every filling multicurve on a smooth surface determines a norm on the first homology group of the surface.
The unit ball of the dual norm is the convex hull of finitely many integer points.
We give an interpretation of these points in terms of certain coorientations of the multicurve.
Our main result is a classification statement: when the surface is hyperbolic and the filling multicurve is geodesic, integer points in the interior of the unit ball of the dual norm classify isotopy classes of Birkhoff sections for the geodesic flow (on the unit tangent bundle to the surface) whose boundary is the symmetric lift of the multicurve.
All results remain true when one replaces the hyperbolic surface by a 2-dimensional orientable hyperbolic orbifold.
\end{abstract}

\section*{Introduction}

This paper deals with the topological study of non-singular flows on 3-manifolds.
With this goal, we study a family of norms on the first homology group of surfaces that may be of independent interest.

Given a smooth 3-manifold~$M$ and a smooth, non-vanishing vector field~$X$ on~$M$, we denote by~$\fxt$ the flow induced by~$X$ on~$M$.
An \dupeterm{embedded Birkhoff section} for $(M,\fxt)$ is a compact, oriented surface~$S$ with boundary, embedded in~$M$, whose interior is positively transverse to~$X$, whose boundary $\partial S$ is tangent to~$X$, and such that every orbit of~$\fxt$ intersects~$S$ in a uniformly bounded time.
On the topological side, an embedded Birkhoff section induces an open book decomposition of the underlying 3-manifold, where the binding is the boundary~$\partial S$, and the fibration of the complement over~$\Sph^1$ is given by an appropriate renormalisation of the flow.
On the dynamical side, when a flow admits an embedded Birkhoff section, its dynamics is encoded by the first-return
map on the section---a much simpler data.
Such a section can be very helpful for understanding some properties of the flow, like the existence or abundance of periodic orbits~\cite{birkhoff1913proof, franks1988generalizations}.

There are several existence results on Birkhoff sections for different classes of flows, for example geodesic flows~\cite{birkhoff1917dynamical, dehornoy2019almost}, Anosov flows~\cite{fried1983transitive}, or Reeb flows~\cite{hryniewicz2020note, hryniewicz2011existence, contreras2022existence, colin2022generic}, among others. On the other hand, as far as we know, there are very few situations in which \emph{all} Birkhoff sections are classified.
An exception is given by the Hopf flow on~$\Sph^3$ where the Birkhoff sections can be explicitely constructed~\cite{dehornoy2022vector}, and the geodesic flow on a flat torus~\cite{dehornoy2015geodesic}.
Our main goal is to provide such a classification, for the geodesic flow on the unit tangent bundle of a hyperbolic surface.

\begin{introtheorem}\label{T:Classification}
For $\Sigma$ a hyperbolic surface and $\gamma$ a finite collection of closed geodesics that fills~$\Sigma$, denote by~$\symgamma$ the symmetric lift of~$\gamma$ in~$\U\Sigma$.
Then there is a one-to-one correspondence between
\vspace{-1mm}
\begin{itemize}
\item[-] isotopy classes of embedded Birkhoff sections for the geodesic flow on the unit tangent bundle~$\U\Sigma$ bounded by the symmetric lift~$\vecgamma$ of~$\gamma$, with negative orientation,
\item[-] points satisfying a certain mod 2 condition in the open dual unit ball~$\BBx\subset\H^1(\Sigma,\ZZ)$ of the intersection norm~$\xx$ associated to~$\gamma$.
\end{itemize}
\end{introtheorem}

Let us mention that such a statement is not really a surprise: the fact that Birkhoff sections with a given boundary up to isotopy correspond to integral points inside certain polyhedrons follows from theorems of Schwartzman, Thurston and Fried, as we explain later in this introduction.
The main contribution of the paper lies in the explicit and combinatorial aspects of all constructions involved.

In the rest of the introduction, we first explain Theorem~\ref{T:Classification} by presenting intersection norms, their dual unit balls and the connection with Eulerian coorientations.
The one-to-one correspondence in Theorem~\ref{T:Classification} is made explicit using a construction we call \emph{Birkhoff--Brunella surfaces} and which is encapsulated in Proposition~\ref{T:Brunella}.
Then we put Theorem~\ref{T:Classification}  in perspective, by connecting it with Thurston--Fried's theory of fibered faces of the Thurston norm ball and with Schwartzman--Fuller--Fried--Sullivan's theory of global sections for flows.

\medskip

\subsection*{Intersection norms.}
Let $\Sigma$ be a smooth surface without boundary.
A \term{multicurve}\footnotemark on~$\Sigma$ is a proper,\\\noindent
\begin{minipage}[t]{0.75\textwidth}
smoothly
immersed 1-submanifold in~$\Sigma$, without boundary, and in general position (meaning that all multiple points are double points where the intersection is transverse).
On a compact surface, a multicurve consists of finitely many closed curves.
\end{minipage}
\begin{minipage}[t]{0.25\textwidth}
  \begin{picture}(0,0)
  \put(03,-15)  {\includegraphics[width=27mm]{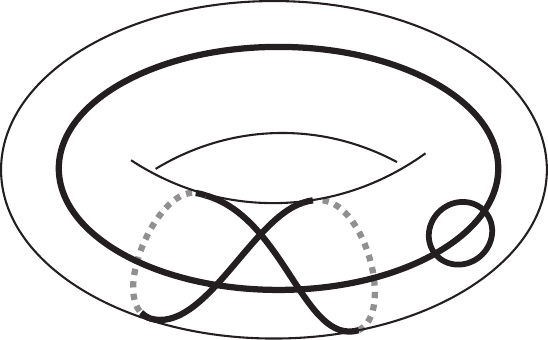}}
  \end{picture}
\end{minipage}
\footnotetext{These are called \emph{divides} by Norbert A'Campo~\cite{acampo1998generic} who, along with Sabir M.~Gusein-Zade, studied divides on the disc in the context of singularities~\cite{acampo1975groupe, gusein-zade1974dynkin, husein-zade1977monodromy}.
They were later generalized to arbitrary surfaces by Masaharu Ishikawa~\cite{ishikawa2004tangent}.
This terminology is maybe not so common in the worlds of surface topologists or dynamists, so we use the more common term \emph{multicurve}.}

Let $\gamma$ be a fixed multicurve on a compact surface $\Sigma$.
We think of $\gamma$ as the discrete analogue of a Riemannian (or Finsler) metric; see \cite{cossarini2018discrete} for a precise connection.
The \term{length} of a generic path $\alpha$ with respect to $\gamma$, denoted $\Len_\gamma(\alpha)$, is defined as the number of crossings between $\alpha$ and $\gamma$, and the length of a homology class $a \in \H_1(\Sigma; \ZZ)$, denoted $\xx(a)$, is the minimal length of a generic integral 1-chain representing the class $a$; see Section~\ref{S:IntersectionNorm}.

Our first result was proven by Schrijver on the torus~\cite{schrijver1993graphs} and stated by Turaev without proof~\cite[Remark 1.9]{turaev2002norm}.

\begin{introprop}
\label{T:NORM}
Let $\Sigma$ be an oriented compact smooth surface and $\gamma$ a multicurve on~$\Sigma$.
Then the function~$\xx: \H_1(\Sigma; \ZZ) \to \NN$ is a symmetric seminorm, that is, it is
\begin{itemize}
\item positively homogeneous: $\xx(n \cdot a) = n\,\xx(a)$ for all $a \in \H_1(\Sigma; \ZZ)$ and $n \in \NN$,
\item subadditive: $\xx(a+b) \leq \xx(a) + \xx(b)$ for all $a,b \in \H_1(\Sigma; \ZZ)$.
\item symmetric: $\xx(-a)=\xx(a)$ for all $a \in \H_1(\Sigma; \ZZ)$,
\end{itemize}
Furthermore, if the multicurve~$\gamma$ is \term{filling} (i.e. it meets every noncontractible closed curve in $\Sigma$), then~$\xx$ is positive definite: $\xx(a) > 0$ if $a \neq 0$.
\end{introprop}

The function~$\xx$ is called the \dupeterm{intersection seminorm} (or \dupeterm{intersection norm} if it is positive definite) associated to~$\gamma$.

\begin{remark}\label{Rk:mod2}
This seminorm satisfies $\xx(a) \equiv [\gamma]_2(a) \mod 2 $ for each $a\in \H_1(\Sigma; \ZZ)$, where $[\gamma]_2$ is the $\ZZ_2$-cohomology class of the cochain that maps each generic smooth 1-chain $\alpha$ to its modulo 2 number of intersections with $\gamma$.
\end{remark}

By a theorem of Thurston~\cite{thurston1986norm}, any integer-valued seminorm $N$ on a \term{lattice}~$L$ (i.e. an abelian group isomorphic to $\ZZ^d$ for some $d \in \NN$) can be written in the form
$\displaystyle N(v) = \max_{\varphi \in F}\varphi(v)$ where $F$ is a finite family of group morphisms $L \to \ZZ$.
In fact, one can take as $F$ the dual unit ball of $N$, that is, the set
$B^*_N$ of homomorphisms $\varphi:L\to \ZZ$ that satisfy $\varphi(v)\leq N(v)$ for all $v \in L$.
Furthermore, if $N$ coincides modulo $m$ (for a certain integer $m \geq 1$) with a given homomorphism $\mu:L \to \ZZ_m$, then one can restrict $F$ to those functionals in $B^*_N$ that coincide with $\mu$ modulo $m$ (see Theorem~\ref{thm:thurston_congruent}).
A natural question is whether these homomorphisms have a nice interpretation in the case that $N$ is an intersection norm (with $m=2$ and $\mu=[\gamma]_2$).
The answer is positive, as we now explain.

Consider a fixed multicurve~$\gamma$ on a surface~$\Sigma$.
A \dupeterm{coorientation} of~$\gamma$ is a continuous transverse orientation defined on $\gamma$ except at the double points, where the coorientation is allowed to flip.
A coorientation $\eta$ induces a cochain $c_\eta$ which maps each generic piecewise-smooth path $\alpha$ in $\Sigma$ to the signed number of crossings of~$\alpha$ with $\gamma$, where the sign of each crossing is determined by $\eta$.
The coorientation~$\eta$ is \dupeterm{Eulerian} if $c_\eta$ is a closed cochain, that is, if $c_\eta(\alpha)=0$ whenever $\alpha$ is a contractible closed curve.
(Equivalently, $\eta$ is Eulerian if around each double point $p$ of~$\gamma$, among the four fragments of $\gamma$ that meet at $p$ there are exactly two that are positively cooriented and two that are negatively cooriented.
See an example on Figure~\ref{F:Intro} left.)
It follows that an Eulerian coorientation $\eta$ induces an integral cohomology class~$[\eta] := [c_\eta] \in \H^1(\Sigma; \ZZ)$. 
Note that different Eulerian coorientations may yield the same cohomology class.

The next result was proven when~$\Sigma$ is a torus by Schrijver, using different methods~\cite[Thm 9]{schrijver1992circuits}.
It is illustrated on Figure~\ref{F:Intro}.

\begin{introtheorem}\label{T:Coor}
Let $\gamma$ be a multicurve on an orientable closed compact surface~$\Sigma$.
Then the cohomology classes in the closed dual unit ball $\overline{\BBx}$ that coincide modulo~2 with $[\gamma]_2$ are precisely the cohomology classes of the Eulerian coorientations of~$\gamma$.
Therefore, for every $a$ in~$\H_1(\Sigma; \ZZ)$ we have
\[\xx(a) = \max_{\substack{\eta~\mathrm{Eulerian}\\\mathrm{coorientation~of~}\gamma}} [\eta](a). \]
\end{introtheorem}

This result also gives an effective way for computing the norm~$\xx$, since it reduces the minimisation over an infinite number of curves to a maximisation over a finite number of coorientations.

Going back to the case where the multicurve $\gamma$ is a geodesic in a hyperbolic surface, Theorem~\ref{T:Classification} states that there is a correspondence between (certain) integral points in the interior of~$\BBx$ and Birkhoff sections for the geodesic flow, and Theorem~\ref{T:Coor} states that (certain) integral points in~$\overline{\BBx}$ can be represented by Eulerian coorientations.
The correspondence of Theorem~\ref{T:Classification} is made explicit by associating to every Eulerian coorientation a certain surface in~$\U\Sigma$, as in the following statement.
The superscript \emph{BB} stands for Birkhoff--Brunella.

\begin{introprop}\label{T:Brunella}
Let $\Sigma$ be a compact oriented surface with a Riemannian metric and~$\gamma$ a finite collection of closed geodesics on~$\Sigma$.
There is a canonical a map~$\BB$ that associates to every Eulerian coorientation~$\eta$ of~$\gamma$ an oriented surface~$\Snu$ in~$\U\Sigma$ whose interior is positively transverse to the geodesic flow and whose oriented boundary is~$-\vecgamma$ .
The Euler characteristic of~$\BB(\eta)$ is independent of~$\eta$ and equals minus twice the number of double points of~$\gamma$.

If two Eulerian coorientations~$\eta_1, \eta_2$ of~$\gamma$ are cohomologous and their common class lies in the interior of~$\BBx$, their interiors are isotopic along the flow. 
\end{introprop}

\medskip

\subsection*{Thurston norm balls, their fibered faces, and suspension flows.}
We now present Thurston's theory of norms and fibered faces for 3-manifolds.
This puts in perspective and explains Theorem~\ref{T:Classification} at an abstract level. 

Given a compact 3-manifold~$M$ with toric boundary, its \emph{Thurston norm}~$\xTh$ is a function on the space $\H_2(M, \partial M; \RR)$ that encodes the minimal negative part of the Euler characteristic of embedded surfaces in~$M$ with boundary in~$\partial M$ in the considered homology class~\cite{thurston1986norm}.
It is a seminorm, and as such it is determined by its unit ball~$B_\xTh$.
The latter turns out to be a polyhedron, which is compact when~$M$ is atoroidal.
It is a topological invariant that is in general hard to compute~\cite{fried2015thurston, agol2020certifying}.

Intersection norms can be seen as 2-dimensional siblings of the Thurston norms since they are defined by minimizing a certain complexity measure over homology classes.
Their unit balls are also polyhedrons, but, unlike unit balls of Thurston norms, these can be easily computed using Theorem~\ref{T:Coor}.

The top-dimensional faces of Thurston norm balls are of two types, namely fibered and non-fibered.
A \emph{fibered face} is such that every integral point in the cone generated by the fibered face is the class of the fibers of a fibration of~$M$ over the circle.

Fried showed~\cite{fried1979fibrations} that every pseudo-Anosov flow~$\ft$ (see Section~\ref{S:Anosov} for a definition) on~$M$ that is tangent to~$\partial M$ and that admits a global cross section canonically determines a fibered face of~$B_\xTh$ as follows:
denote by~$D_\varphi$ the convex cone generated by the homology classes of the periodic orbits of~$\ft$ in~$\H_1(M;\RR)$.
This cone actually coincides with the cone over the set of Schwartzman asymptotic cycles~\cite{schwartzman1957asymptotic}.
The dual cone~$C_\varphi$ in~$\H_2(M, \partial M; \RR)$ is defined as those classes that pair positively with all of~$D_\varphi$.
It turns out that the integral classes in~$C_\varphi$ correspond exactly to the classes of the global sections to~$\ft$.
Therefore~$C_\varphi$ is exactly the cone over the interior of a fibered face of the Thurston norm ball.
The cones $D_\varphi$ and $C_\varphi$ are polyhedric, and Fried also gives an algorithm~\cite{fried1982geometry} for computing~$D_\varphi$ and~$C_\varphi$ starting from a Markov partition for~$\ft$.

The connection with Birkhoff sections can be made as follows:
assume that $\beta$ is a collection of periodic orbits of a flow~$\varphi$ in~$M$, one can blow up the link~$\beta$ and obtain a 3-manifold~$\ov{M\setminus\beta}$ with toric boundary~$\partial \ov{M\setminus\beta}$.
If~$\ft$ is of class $C^1$, then it extends to a non-singular flow~$(\varphi_\beta^t)_{t\in\RR}$ on~$\ov{M\setminus\beta}$.
If~$\ft$ was of Anosov or pseudo-Anosov type, then $(\varphi_\beta^t)_{t\in\RR}$ is pseudo-Anosov.
In this context a Birkhoff section for~$\ft$ with boundary in~$\beta$ extends to a global section for the flow $(\varphi_\beta^t)_{t\in\RR}$.
The discussion of the previous paragraph then implies that, if $\beta$ bounds a Birkhoff section, isotopy classes of Birkhoff sections whose boundary is in~$\beta$ are classified by integral points in a certain polyhedral cone~$C_{\varphi,\beta}$ in~$\H_2(\ov{M\setminus\beta}, \partial \ov{M\setminus\beta}; \RR)\simeq\H_2(M,\beta; \RR)$.

In the context of Theorem~\ref{T:Classification}, $M$ is the unit tangent bundle~$\U\Sigma$ to a hyperbolic surface~$\Sigma$, $\ft$ is the geodesic flow on~$\U\Sigma$, and~$\beta$ is the symmetric lift~$\vecgamma$ of a filling collection~$\gamma$ of geodesics on~$\Sigma$, see Section~\ref{S:Geodesic} for the definitions.
The set of Birkhoff sections for the geodesic flow bounded by~$\vecgamma$ is then the cone over a fibered face of the Thurston norm ball in~$\HSgammaZ$ that we denote by~$C_{\mathrm{geod}, \vecgamma}$.

In Theorem~\ref{T:Classification}, the assumption that the oriented boundary is exactly~$-\vecgamma$ (that is, every boundary component has multiplicity~$-1$) can be seen as a restriction on the homology class of the section: it has to lie in a certain affine subspace denoted by~$\partial^{-1}_{\vecgamma}(-1, \dots, -1)$ of~$\HSgammaR$.
This means that the Birkhoff sections we are interested in are enumerated by the intersection of the cone~$C_{\mathrm{geod}, \vecgamma}$ with the affine subspace~$\partial^{-1}_{\vecgamma}(-1, \dots, -1)$.
It turns out that a suitable choice of an origin identifies the latter with $\H_1(\Sigma; \RR)$, see Section~\ref{S:Affine}.
Under this identification, Theorem~\ref{T:Classification} can be summarized by the equality
\[C_{\mathrm{geod}, \vecgamma}\cap\partial^{-1}_{\vecgamma}(-1, \dots, -1) = \frac 12\BBx. \]
Our paper adds to this description the elementary and explicit characters of all the involved constructions.
Indeed, as far as we know, there is no other Anosov or pseudo-Anosov flow for which the set of global or Birkhoff cross sections admits such an explicit and combinatorial description.

\begin{figure}[h!]
\centering
\includegraphics[width=.75\textwidth]{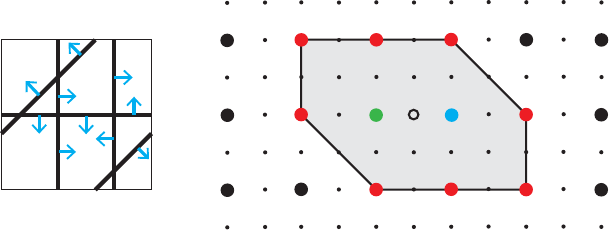}
\caption{Illustration of Theorems~\ref{T:Classification} and~\ref{T:Coor} in the case of $\Sigma$ a torus (with an abuse since Theorem~\ref{T:Classification} deals with higher-genus surfaces, whose homology has dimension $\ge 4$). On the left, a multicurve~$\gamma$ on~$\Sigma$ consisting of four geodesics, and an Eulerian coorientation (blue arrows).
Seen as a graph, $\gamma$ has 5 vertices and 10 edges.
On the right, the dual unit ball~$\BBx$ 
of the associated intersection norm.
The empty circle denotes the origin.
The big dots denote those classes in~$\H^1(\Sigma; \ZZ)$ congruent to~$[\gamma]_2$ mod 2.
Among these classes, 10 (in blue, green and red) are in the dual unit ball~$\BBx$ and correspond to all cohomology classes of Eulerian coorientations of~$\gamma$ (Theorem~\ref{T:Coor}).
For example, the class corresponding to the blue coorientation is the blue point.
The blue and green points lie in the interior of~$\BBx$, hence describe two isotopy classes of Birkhoff cross sections for the geodesic flow 
bounded by~$-\symgamma$.
If the genus of $\Sigma$ was at least 2, there would be no other isotopy class of Birkhoff cross section for the geodesic flow  (Theorem~\ref{T:Classification}).
The 8 red points are on the boundary of~$\BBx$ and correspond to classes of surfaces transverse to 
the geodesic flow, but not intersecting every orbit, and bounded by~$-\symgamma$.
}
\label{F:Intro}
\end{figure}

\begin{remark}	
One may wonder how general Theorem~\ref{T:Classification} is, namely whether one can hope for an analog statement for any (transitive) Anosov flow.
As explained above, the set of Birkhoff sections up to isotopy fixing the boundary is described by the integral points inside a certain polyhedron.
However we do not know how to describe this polyhedron in general.
It seems to be related to linking numbers of periodic orbits of the flow~\cite{dehornoy2015geodesic, dehornoy2017which}.
However linking numbers are only defined for null-homologous links.
Ghys proved that Gauss linking forms describe all linking numbers between periodic orbits for a vector field in a homology sphere~\cite{ghys2009right}.
Moreover he showed how to use these Gauss forms to decide whether all finite collections of periodic orbits bound a Birkhoff sections (which he calls \emph{left-} or \emph{right-handed} flows).
Probably one should first extend the concept of Gauss linking forms to manifolds that are not rational homology spheres, and see how this helps defining linking of periodic orbits and more generally of invariant measures.
Then one could hope that these generalized linking describe exactly the homological information needed to apply Schwartzman's criterion, as we will do in Section~\ref{S:Birkhoff}.
\end{remark}

\begin{remark}
It may look strange to deal with Birkhoff cross sections with negative boundary and not with positive ones, \emph{i.e.,} with surfaces such that the orientation of the boundary inherited from the orientation of the surface (itself inherited from the coorientation of the interior surface by the flow) is opposed to the direction of the flow.
The reason is that there is actually no positive Birkhoff cross section for the geodesic flow, as explained in Théo Marty's thesis~\cite[Chap~3]{marty2021anosov}.
One could then look at \emph{mixed} sections, namely transverse surfaces some of whose boundary components are positively tangent to the geodesic flow and some others are negatively tangent.
There are more mixed sections than negative.
Alas, we have no analog of Proposition~\ref{T:Brunella} in this more general case, meaning that we do not have an elementary way to construct all mixed sections.
\end{remark}

\begin{remark}The case of the torus with a flat metric is not covered by Theorem~\ref{T:Classification}.
In this case, the fact that the unit tangent bundle~$\U\mathbb{T}^2$ is trivial allows to cut-and-glue horizontal tori to Birkhoff cross sections, so that there are infinitely many isotopy classes with a given boundary.
However, modulo this additional operation, there are still only finitely many classes.
These have been classified in a previous work by the second author~\cite[Thm 3.12]{dehornoy2015geodesic}.
The statement is similar, namely equivalence classes of Birkhoff sections are classified by points in the interior of a certain polygon with integral vertices.
The statement is even more general since, in this restricted case of the torus, there is no assumption that the boundary of the section is symmetric.
One could recover this earlier result in the symmetric case by a proof very similar to that of Theorem~\ref{T:Classification}.
\end{remark}

\medskip

\subsection*{Extension to 2-dimensional orbifolds.}
The results of this paper can be generalized in the following sense.
Instead of considering orientable surfaces only, one can consider orientable 2-dimensional orbifolds, as introduced by Thurston~\cite{thurston1980geometry}.
Such a 2-orbifold~$\OO$ is described by an orientable topological surface~$\Sigma_\OO$ and charts that are local homeomorphisms $\RR^2/(\ZZ/k\ZZ)\to\Sigma_\OO$, where $\ZZ/k\ZZ$ acts by rotation on~$\RR^2$.

There are several possible definitions for the homology of an orbifold that yield different spaces.
The one that is useful here is the most elementary: we define~$\H_i(\OO; \RR)$ to be the space $\H_i(\Sigma_\OO; \RR)$.
In this context the definition of intersection norms extends trivially.
Proposition~\ref{T:NORM} and Theorem~\ref{T:Coor} still hold.
Now the unit tangent bundle~$\U\OO$ is 3-manifold that is a Seifert fibered space over~$\Sigma_\OO$.
The geodesic flow is well defined on~$\U\OO$, and when $\OO$ is hyperbolic it is still of Anosov type.
Proposition~\ref{T:Brunella} extends directly in this context.
Concerning Theorem~\ref{T:Classification}, it has to be modified for taking into account orbifolds that are homology spheres ---a case that does not occur with hyperbolic surfaces.

\begin{introtheorem}\label{T:ClassificationBis}
Let $\OO$ be a hyperbolic orientable 2-dimensional orbifold.
Let $\gamma$ be a finite collection of closed geodesics on~$\OO$.

\begin{itemize}
\item
If $\Sigma_\OO$ is a sphere, then $\U\OO$ is a rational homology sphere.
In this situation, the link~$-\vecgamma$ bounds a Birkhoff section for the geodesic flow in~$\U\OO$ if and only if $\gamma$ is filling in~$\Sigma_\OO$.
In that case, the Birkhoff section is unique up to isotopy fixing the boundary.
\item
If $\Sigma_\OO$ is not a sphere and if $\gamma$ is filling, then the map~$[\eta]\mapsto \{\Snu\}$ is a one-to-one correspondence between integer points in the open unit ball~$\mathrm{int}(\BBx)$ congruent to~$[\gamma]_2\mathrm{~mod~}2$ and isotopy classes of Birkhoff cross sections for the geodesic flow in~$\U\Sigma$ with boundary~$-\vecgamma$.
\item
If $\Sigma_\OO$ is not a sphere and $\gamma$ is not filling, then there is no surface bounded by~$-\vecgamma$ and transverse to the geodesic flow.
\end{itemize}
\end{introtheorem}

A particular case is when $\OO$ is a hyperbolic triangular orbifold, that is, a sphere with three conic points.
In this case every collection~$\gamma$ of closed geodesics is filling, hence its lift~$\vecgamma$ bounds a Birkhoff section.
This is a particular case of the main result of~\cite{dehornoy2017which} which proves that in this case every finite collection of periodic orbits (even non-symmetric) bounds a Birkhoff section for the geodesic flow.

\medskip

\subsection*{Acknowledgments.}
Pierre D thanks \'Etienne Ghys and Adrien Boulanger for related discussions, and Elena Kudryavtseva who initiated this article by asking several questions about Birkhoff sections.
The authors thank the anonymous referees for useful suggestions, in particular the extension to orbifolds.

\setcounter{tocdepth}{2}
\tableofcontents

\section{Intersection norms and proof of Proposition~\ref{T:NORM}}
\label{S:IntersectionNorm}

In the whole section we fix an oriented compact smooth surface~$\Sigma$ with empty boundary and a multicurve~$\gamma$ on~$\Sigma$.
(Recall that a multicurve in $\Sigma$ is a compact, closed 1-manifold that is smoothly immersed in $\Sigma$, self-transverse, and has no points of multiplicity $>2$).

\begin{defi}\label{D:Paths}
A \term{path} in $\Sigma$ is a continuous function $\alpha: I \to \Sigma$ (where $I\subseteq \RR$ is a compact interval), considered up to a uniform shift in the parametrization,
so that the concatenation $\alpha \beta$ of two consecutive paths $\alpha,\beta$ is well defined.
The \term{reverse} of a path $\alpha$ is the path 
$\alpha^\dag(t) = \alpha(-t)$.
The \term{trivial path} at a point $p \in \Sigma$ is denoted $1_p$. 
\end{defi}

\begin{defi}\label{D:Length}
A smooth path in $\Sigma$ is \term{generic} (with respect to $\gamma$) if it has no endpoint on $\gamma$, it is transverse to $\gamma$, and it avoids the double points of $\gamma$.
We denote by~$P_\gamma$ the set of \term{generic piecewise-smooth paths}, obtained by concatenating finitely many generic smooth paths.
The \term{length} with respect to $\gamma$ of a path $\alpha \in P_\gamma$ is the number of times that it meets $\gamma$,
\[ \Len_\gamma(\alpha) = | \alpha^{-1}(\gamma) |. \]
\end{defi}

\begin{defi}\label{D:Norm}
A \term{generic integral 1-chain} is a linear combination $\alpha = \sum_i c_i \alpha_i$ of paths $\alpha_i \in P_\gamma$ with integer coefficients $c_i \in \ZZ$.
Its \term{length} is defined as
\[ \Len_\gamma(\alpha) = \sum_i |c_i| \, \Len_\gamma(\alpha_i). \]
Note that every homology class~$a$ in~$\H_1(\Sigma; \ZZ)$ may be represented by a generic integral 1-chain.
The length of the homology class~$a$ is defined as
\[ \xx(a) = \min_{\substack{\alpha\text{ closed generic integral} \\\text{1-chain such that }[\alpha] = a} } \Len_\gamma( \alpha ). \]
A closed generic 1-chain that minimizes length in its homology class is called an~$\xx$-\term{realizing 1-chain}.
The function~$\xx:\H_1(\Sigma; \ZZ) \to \NN$ is called the \term{intersection seminorm} (or \term{intersection norm}, if it is positive definite) associated to~$\gamma$.
\end{defi}

\begin{figure}
\centering
\def\svgwidth{.8\linewidth}
\begingroup%
  \makeatletter%
  \providecommand\color[2][]{%
    \errmessage{(Inkscape) Color is used for the text in Inkscape, but the package 'color.sty' is not loaded}%
    \renewcommand\color[2][]{}%
  }%
  \providecommand\transparent[1]{%
    \errmessage{(Inkscape) Transparency is used (non-zero) for the text in Inkscape, but the package 'transparent.sty' is not loaded}%
    \renewcommand\transparent[1]{}%
  }%
  \providecommand\rotatebox[2]{#2}%
  \newcommand*\fsize{\dimexpr\f@size pt\relax}%
  \newcommand*\lineheight[1]{\fontsize{\fsize}{#1\fsize}\selectfont}%
  \ifx\svgwidth\undefined%
    \setlength{\unitlength}{446.00001526bp}%
    \ifx\svgscale\undefined%
      \relax%
    \else%
      \setlength{\unitlength}{\unitlength * \real{\svgscale}}%
    \fi%
  \else%
    \setlength{\unitlength}{\svgwidth}%
  \fi%
  \global\let\svgwidth\undefined%
  \global\let\svgscale\undefined%
  \makeatother%
  \begin{picture}(1,0.44581387)%
    \lineheight{1}%
    \setlength\tabcolsep{0pt}%
    \put(0,0){\includegraphics[width=\unitlength,page=1]{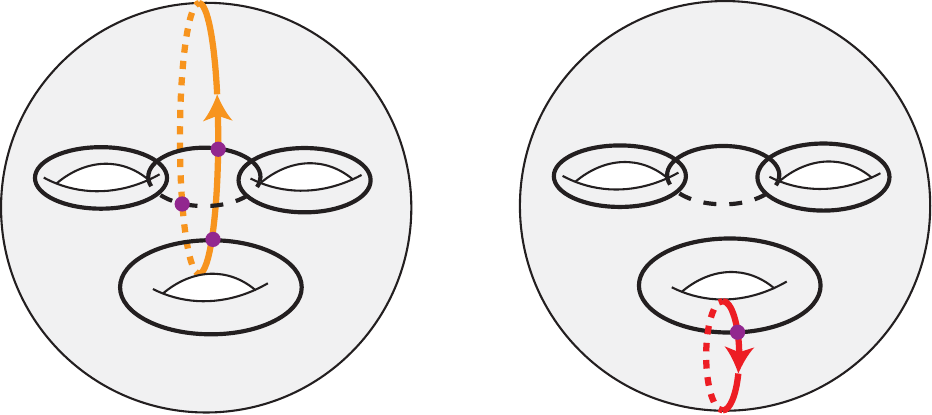}}%
    \put(0.24023557,0.366923){\color[rgb]{0,0,0}\makebox(0,0)[lt]{\lineheight{1.25}\smash{\begin{tabular}[t]{l}$\alpha_1$\end{tabular}}}}%
    \put(0.80228528,0.03622544){\color[rgb]{0,0,0}\makebox(0,0)[lt]{\lineheight{1.25}\smash{\begin{tabular}[t]{l}$\alpha_2$\end{tabular}}}}%
  \end{picture}%
\endgroup%

\caption{\small A genus~3 surface with a multicurve~$\gamma$ made of four closed curves (black).
On the left the curve~$\alpha_1$ (orange and bold) is transverse to~$\gamma$ and intersects it three times.
On the right $\alpha_2$ (red) is homologous to~$\alpha_1$ since their difference bounds a subsurface, namely the right hemisurface.
The curve~$\alpha_2$ intersects~$\gamma$ only once.
This number cannot be reduced to~$0$ in the same homology class, hence $\alpha_2$ is $\xx$-realizing and we have~$\xxx{[\alpha_1]}=\xxx{[\alpha_2]}=|\{\alpha_2^{-1}(\gamma)\}|=1.$}
\label{F:IntGenre3}
\end{figure}

The function~$\xx$ has three properties that make it a seminorm, namely it is positively homogeneous, subadditive and symmetric.
To prove the first point we need two facts about curves on surfaces.
Note first that every homology class $a \in \H_1(\Sigma; \ZZ)$ can be represented by an oriented multicurve.
A multicurve $\alpha$ is~\term{simple} if it has no double points, and is generic (with respect to $\gamma$) if it is transverse to $\gamma$ and the union $\alpha \cup \gamma$ is a multicurve. (The last condition holds if and only if each of the two multicurves $\alpha,\gamma$ avoids the double points of the other one.)

\begin{lemma}[Simplification]\label{L:Simplification}
Every homology class~$a$ in~$\H_1(\Sigma; \ZZ)$ can be represented by a simple oriented multicurve that is generic with respect to $\gamma$ and $\xx$-realizing.
\end{lemma}

\begin{proof}
Let $\alpha$ be an oriented multicurve that represents the class $a$ and is generic with respect to $\gamma$ and $\xx$-realizing.
To make $\alpha$ simple, we eliminate each self-crossing of $\alpha$ by performing a local modification of the form $\vcenter{\hbox{\includegraphics[width=4em]{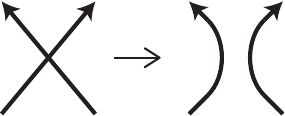}}}$.
\end{proof}

\begin{lemma}[Partitioning]\label{L:Partitioning}
Every simple oriented multicurve in $\Sigma$ of homology class $n \cdot a$
(for some $a \in \H_1(\Sigma; \ZZ)$ and $n \in \NN_{\neq 0}$)
is a union of $n$ disjoint simple oriented multicurves, each of class $a$.
\end{lemma}

\begin{proof} Let~$\beta$ be a simple oriented multicurve of homology class~$n\cdot a$.
Since $\beta$ is of class $n\cdot a$, its algebraic number of crossings with any generic oriented loop is a multiple of $n$.
Therefore we can label the regions (i.e. connected components) of~$\Sigma\setminus\beta$ with integers modulo $n$ in such a way that the label increases by~1 when one crosses~$\beta$ positively (i.e. from right to left).
For every $i\in\ZZ/n\ZZ$, denote by~$\alpha_i$ the union of those components of $\beta$ that have regions labelled~$i$ on their right, and regions labelled~$i+1$ on their left. 
Every~$\alpha_i$ is a simple multicurve, and by conctruction~$\beta$ is the union of all of them.
Any two curves $\alpha_i$, $\alpha_j$ are homologous since $\alpha_i-\alpha_j$ bounds a subsurface of~$\Sigma$ (namely, the part with labels in $[i,j)$).
This implies that $[\beta] = n \cdot [\alpha_i]$ for every $i$, and since $\H_1(\Sigma; \ZZ)$ has no torsion, we conclude that $[\alpha_i]=a$.
\end{proof}

Now let us show that $\xx$ is a seminorm.
The symmetry property $\xx(-a)=\xx(a)$ is evident since the number of intersections does not change by reversing the orientation of a curve.
We have to prove positive homogeneity and subadditivity.

\begin{lemma}[Positive homogenity]\label{L:Positive_homogeneity}
For every~$a$ in~$\H_1(\Sigma; \ZZ)$ and for all $n\in\NN$ one has
\[ \xxx{n \cdot a} = n \,\xxx{a}. \]
\end{lemma}

\begin{proof}
Given~$a \in \H_1(\Sigma; \ZZ)$ and $n\in \NN$, consider a realizing multicurve $\alpha$ in~$a$.
Since $n$ parallel copies of~$\alpha$ intersect~$\gamma$ at $n\,\xxx{a}$~points, we have~$\xxx{n\cdot a}\le n\,\xxx{a}$.
For the reverse inequality, consider an $\xx$-realizing multicurve~$\beta$ of homology class~$n\cdot a$.
By simplification (Lemma~\ref{L:Simplification}) we can suppose~$\beta$ simple, and then it follows by partitioning (Lemma~\ref{L:Partitioning}) that~$\beta$ is the union of $n$ multicurves $\alpha_i$ of class $a$.
Each multicurve~$\alpha_i$ has at least $\xxx{a}$ intersections with $\gamma$, which implies that $\beta$ has at least $n \, \xxx{a}$ intersections with~$\gamma$, proving the inequality~$\xxx{n\cdot a}\geq n\,\xxx{a}$.
\end{proof}

\begin{lemma}[Subadditivity]\label{L:Subadditivity}
For every~$a, b$ in~$\H_1(\Sigma; \ZZ)$ one has
\[\xxx{a+b}\le\xxx{a}+\xxx{b}.\]
\end{lemma}

\begin{proof}
The union of two multicurves that realize~$\xxx{a}$ and $\xxx{b}$ crosses~$\gamma$ in~$\xxx{a}+\xxx{b}$ points, giving $\xxx{a+b}\le\xxx{a}+\xxx{b}$.
\end{proof}

This finishes the proof of Proposition~\ref{T:NORM} which states that the function~$\xx$ is a seminorm on $\H_1(\Sigma; \ZZ)$.

\begin{remark}
One can easily extend the notion of intersection norm to a surface with boundary $\Sigma$, by allowing the multicurves $\gamma$ to contain arcs with endpoints on~$\partial \Sigma$ (as did A'Campo~\cite{acampo1998generic, acampo1975groupe}).
One then obtains two norms on~$\H_1(\Sigma; \ZZ)$ and $\H_1(\Sigma, \partial\Sigma; \ZZ)$, depending on whether one considers absolute or relative homology classes.
Proposition~\ref{T:NORM} also holds in the second context.
\end{remark}

\begin{remark}\label{R:StableNorm}
One can wonder how the intersection norms compare with other known norms on the first homology of a surface.
For example, the \emph{stable norm}~${\bf x}_g$ induced by a metric $g$, is defined by $\displaystyle{{\bf x}_g(a)=\liminf_{n\to\infty}\min_{\alpha^{(n)}\in na}{g(\alpha^{(n)})}/{n}}$.
On a surface the stabilisation is not necessary, so that one has $\displaystyle{{\bf x}_g(a)=\min_{\alpha\in a}g(\alpha)}$.
One can check that if $(\gamma_k)_{k\in\NN}$ is a sequence of filling geodesics that approximates~$g$, meaning that the sequence of invariant measures on~$\U\Sigma$ that are concentrated on the lift~$\vec\gamma_k$ tends in the weak-* sense to the Liouville measure defined by~$g$ on~$\U\Sigma$, then the rescaled norms~$\frac1{g(\gamma_k)}{\bf x}_{\gamma_n}$ tend to the stable norm of~$g$.
Equivalently, the rescaled unit balls~$g(\gamma_k) B_{{\bf x}_{\gamma_k}}$ tend to the unit ball of the stable norm.
\end{remark}

\section{Unit balls and coorientations}
\label{S:Coor}

The context remains the same as in the previous section: we fix an oriented closed compact smooth surface~$\Sigma$ of genus at least~$1$ and a multicurve~$\gamma$ on it.
We have shown that the intersection norm~$\xx$ is an integer-valued seminorm on the lattice $\H_1(\Sigma; \ZZ) \simeq \ZZ^{2g}$.
By Remark~\ref{Rk:mod2} it coincides modulo 2 with $[\gamma]_2$.
Therefore we may apply the following result of Thurston (as extended in Section~\ref{S:integral_seminorms}).
Recall that a \dupeterm{lattice} $L$ is a finitely generated free abelian group.
Its \dupeterm{dual lattice} $L^*$ is the group of homomorphisms $L \to \ZZ$.
Note that $L \simeq L^* \simeq \ZZ^d$ for some $d \in \NN$.

\begin{theo}[Thm. 2 of \cite{thurston1986norm} and Theorem~\ref{thm:thurston_congruent}]\label{T:Thurston}
Every integral seminorm $N$ on a lattice $L$ is of the form
\[ N(v) = \max_{\varphi \in B^*_N} \varphi(v), \]
where $B^*_N \subseteq L^*$ is the \dupeterm{dual unit ball} of $N$, that is the (finite) set of group homomorphisms $\varphi:L \to \ZZ$ that satisfy $\varphi(v) \leq N(v)$ for all $v \in L$.
Furthermore, if $N$ coincides modulo a certain integer $m>1$ with a given homomorphism $\mu : L \to \ZZ_m$, then we have
\[ N(v) = \max_{\substack{\varphi \in B^*_N\\ \varphi_{\mmod m} = \mu }} \varphi(v). \]
\end{theo}

Our goal in this section is to prove Theorem~\ref{T:Coor}, that is, to characterize the points of $\overline{\BBx}$ that coincide modulo 2 with $[\gamma]_2$.
Specifically, we will show that these cohomology classes are precisely those that can be represented by Eulerian coorientations.
We will do so as follows.

\begin{figure}
\centering
\includegraphics[width=.7\textwidth]{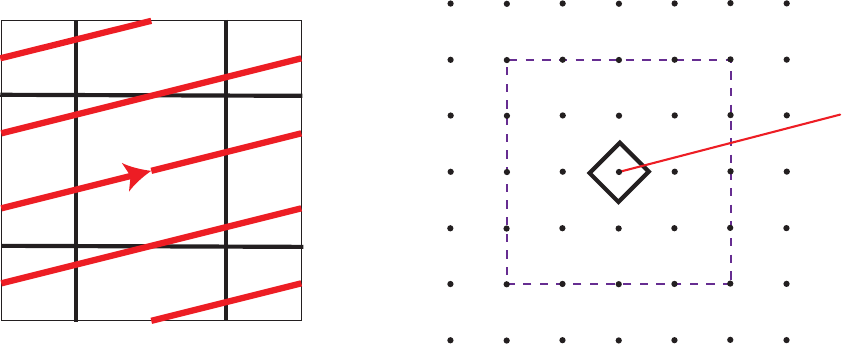}
\caption{\small A torus with a collection~$\gamma$ (black) made of four curves, two vertical and two horizontal.
The curve~$\alpha$ (red and bold) intersects~$\gamma$ in 10 points.
It is the smallest number for a curve whose homology class is~$(4,1)$, so that $\xx(4,1)=10$.
The norm~$\xx$ is actually given by~$\xxx{(p,q)}=2|p|+2|q|$ in the canonical coordinates.}
\label{F:IntTore}
\end{figure}

Recall, from the introduction, that a \dupeterm{coorientation} of~$\gamma$ is a continuous transverse orientation defined on $\gamma$ except at the double points, where the coorientation is allowed to flip.
A coorientation determines a 1-cochain $c_\eta$ which maps each generic piecewise-smooth path $\alpha$ in $\Sigma$ to the signed number of crossings of $\alpha$ with~$\gamma$, where the sign of each crossing is determined by $\eta$.
This cochain clearly satisfies \[
c_\eta(\alpha) \leq   \Len_\gamma(\alpha) \quad\text{and}\quad
c_\eta(\alpha) \stackrel{\mmod 2}{\equiv} \Len_\gamma(\alpha)
\quad\text{for each generic path $\alpha$.} \]
If we think of the multicurve $\gamma$ as a discrete metric, then we can see a coorientation~$\eta$ as a discrete field of unit-norm covectors, and the number $c_\eta(\alpha)$ as the value of the integral of $\eta$ along the path $\alpha$.

A coorientation $\eta$ is \dupeterm{Eulerian} if $c_\eta$ is a closed cochain, that is, if $c_\eta(\alpha)=0$ whenever $\alpha$ is a contractible closed curve.
In this case, the coorientation $\eta$ defines a cohomology class 
$[\eta] := [c_\eta] \in \H^1(\Sigma; \ZZ)$.
This cohomology class $h = [\eta]$ satisfies the properties
\[ h(a) \leq \xx(a) \quad\text{and}\quad
h(a) \stackrel{\mmod 2}{\equiv} \xx(a)
\quad\text{for all } a \in \H_1(\Sigma;\ZZ),\]
and we say then that $h$ is $\gamma$-\term{special}.

To go backwards, from a $\gamma$-special cohomology class $h$ to a coorientation $\eta$ such that $h=[\eta]$, we will rely on an auxiliary object called an \emph{eikonal function}.
An \term{eikonal function} on a surface-with-a-multicurve $(\Sigma,\gamma)$ is a function $f: \Sigma \setminus \gamma$ that satisfies
\begin{equation}\label{eq:eikonal_intro}
| f(y) - f(x) | \leq d_\gamma(x,y)
\quad \text{and} \quad
f(y) - f(x) \stackrel{\mmod 2}{\equiv} d(x,y)
\quad\text{for all } x,y \in \Sigma\setminus\gamma.
\end{equation}
If we think of the multicurve $\gamma$ as a discrete metric, and we see (Eulerian) coorientations as (closed) unitary 1-forms, then we should see eikonal functions as scalar-valued functions that are nonexpansive (or 1-Lipschitz).
We can differentiate an eikonal function to obtain an Eulerian coorientation, and reciprocally, on a simply connected surface, we can integrate an Eulerian coorientation to obtain an eikonal function.

We will use eikonal functions as follows.
Let $\left(\wt\Sigma,\pi\right)$ be the universal cover of $\Sigma$, and let $x_0 \in \Sigma$ be a fixed, arbitrary point.
The surface $\wt\Sigma$ has a multicurve $\wt\gamma = \pi^*\gamma$ (the pullback of $\gamma$ by the covering map $\pi$).
An Eulerian coorientation $\eta$ determines an eikonal function $f_\eta$ on $\wt\Sigma \setminus \wt\gamma$, called the \emph{primitive} of $\eta$, by the formula
\[ f_\eta(x) = c_\eta(\pi \circ \alpha_{x_0,x}),\]
where $\alpha_{x_0,x}$ is a generic path in $\wt\Sigma$ from $x_0$ to $x$.
We note that $f_\eta$ is equivariant with respect to the cohomology class $h=[\eta]$, which means that
\[ f(T_\beta(x)) - f(x) = h[\beta], \]
for any points $x \in \wt \Sigma$, and any loop homotopy class $\{\beta\} \in \Pi_1(\Sigma,\pi(x_0))$, where $T_\beta$ is the automorphism of $\wt\Sigma$ induced by the curve $\beta$.

Moreover, this process can be reversed: any $h$-equivariant eikonal function $f$ can be differentiated to obtain a coorientation $\eta$ of cohomology class $h$, such that $f_\eta = f$.
Therefore, to go backwards, from a $\gamma$-special cohomology class $h$ to a coorientation $\eta$ such that $[\eta]=h$, we do as follows.
We define first an $h$-equivariant function $f: \pi^{-1}(p_0) \to \ZZ$, where $p_0 = \pi(x_0) \in \Sigma$.
We show that~$\wt f$ is pre-eikonal (i.e. it satisfies \eqref{eq:eikonal_intro}, even thought it is not defined at all points) since $h$ is $\gamma$-special.
Finally, we show that any pre-eikonal funcion can be naturally extended, using a standard formula, to an eikonal function $\ov f$ defined on the whole space.
Moreover, this extended function $\ov f$ is $h$-equivariant if $f$ is so.
Differentiating the eikonal funcion $\ov f$ we obtain the coorientation $\eta$ such that $f_\eta = \ov f$, and therefore $[\eta] = h$.

\subsection{Coorientations of multicurves}

Recall that~$\gamma$ is a multicurve in $\Sigma$.
\begin{figure}
\centering
\begin{picture}(45,40)
\put(0,0){\includegraphics[width=40mm]{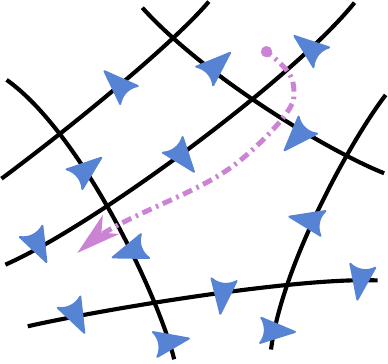}}
\put(23,17){$\alpha$}
\end{picture}
\caption{\small A piece of a multicurve~$\gamma$ (black).
A coorientation~$\eta$ of~$\gamma$ is indicated with blue arrows.
A path~$\alpha$ transverse to~$\gamma$ is shown (purple and dotted).
The pairing $\langle \eta ,\alpha \rangle $ equals $-1+2=+1$ on this example.}
\label{F:Pairing}
\end{figure}
A \term{cross-vector} on $\gamma$ is a vector tangent to $\Sigma$ that is located at a \emph{simple} point of $\gamma$, and is transverse to $\gamma$.
The set of such vectors, considered as a topological subspace of the tangent bundle of $\Sigma$, is denoted $C_\gamma$.
Note that this space has finitely many connected components.

\begin{defi}\label{D:Coor}
An integral \term{cross-functional} on $\gamma$ is a function $\eta: C_\gamma \to \ZZ$ that is locally constant and satisfies the equation $\eta(-v) = -\eta(v)$ for all $v \in C_\gamma$.
A \term{coorientation} of $\gamma$ is a cross-functional with values $\pm 1$.
Note that there are finitely many coorientations of $\gamma$.
\end{defi}

As mentioned in the introduction, each coorientation $\eta$ induces a cochain $c_\eta$ whose cohomology class $[c_\eta] \in \H^1(\Sigma;\ZZ)$ is in the dual unit ball $B^*_{\xx}$, as we will see below.
To reverse this process and show that each cohomology class $h \in B^*_{\xx}$ (equivalent to $\xx$ mod 2) can be represented by a coorientation, we must understand precisely which cochains are induced by coorientations or, more generally, by cross-functionals of $\gamma$.
These cochains are called \emph{cross-cochains}, and are characterized as follows.

Recall from Definition~\ref{D:Length} that $P_\gamma$ is the space of piecewise-smooth paths on $\Sigma$ that are generic with respect to $\gamma$.

\begin{defi}\label{D:cross-cochain}
An integral \term{cross-cochain} (with respect to the multicurve $\gamma$) is a function $c: P_\gamma \to \ZZ$ with the following properties:
\begin{itemize}
\item \emph{additive} with respect to concatenation of paths, that is, such that $c(\delta\eps) = c(\delta) + c(\eps)$ if~$\delta$,~$\eps \in P_\gamma$ are consecutive paths;
\item \emph{alternating} with respect to path reversion, that is, such that $c(\alpha^\dag) = -c(\alpha)$ for all paths $\alpha \in P_\gamma$;
\item \emph{supported on $\gamma$}, that is, such that $c(\alpha) = 0$ if $\alpha$ does not meet $\gamma$;
\item \emph{locally constant}, that is, constant on any continuous family $(\eps_t)_{t\in[0,1]}$ of smooth paths $\eps_t \in P_\gamma$.
(Such a family of paths is not called a homotopy of paths because the endpoints may move.
However, the endpoints never cross $\gamma$, since at the instant of crossing the path would not be in $P_\gamma$.)
\end{itemize}
\end{defi}

\begin{defi}\label{D:Pairing}
The \term{integral} of a cross-functional~$\eta$ along a path $\alpha \in P_\gamma$ is the number
\[ c_\eta(\alpha) := \sum_{t \in \alpha^{-1}(\gamma) } \eta(\alpha'(t)). \]
\end{defi}

\begin{lemma}\label{L:bijection_functional_cochain} The map $\eta \mapsto c_\eta$ is a bijection from the set of integral cross-functionals to the set of integral cross-cochains on $\gamma$.
\end{lemma}

The proof is straightforward.

\begin{proof} For a cross-functional $\eta$, it is clear that $c_\eta$ is a cross-cochain.
Let $F$ be the map from the set of integral cross-functionals to the set of integral cross-cochains given by $F(\eta) = c_\eta $.

To show that $F$ is bijective, we use the following notation.
For a cross-vector $v \in C_\gamma$, let $C_\gamma(v)$ be the connected component of $C_\gamma$ contaning $v$, and let $P_\gamma(v)$ be the set of smooth paths in $P_\gamma$ that cross $\gamma$ exactly once, and with velocity $v'$ in~$C_\gamma(v)$.
Note that any two arcs $\eps_0$, $\eps_1 \in P_\gamma(v)$ are connected by a continuous family of arcs $(\eps_t)_{t\in[0,1]}$ in $P_\gamma(v)$.
This implies that any cross-cochain is constant on $P_\gamma(v)$.

The map $F$ is injective since given two cross-functionals $\eta \neq \eta'$, we see that $c_\eta(v) \neq c_{\eta'}$ by evaluating these two cochains at a path $\gamma \in P_\gamma(v)$, where $v \in C_\gamma$ is a cross-vector such that $\eta(v) \neq \eta'(v)$.

Now let us show that $F$ is surjective.
Given a cross-cochain $c$, we shall produce a cross-functional $\eta$ such that $c_\eta=c$.
We define $\eta$ as follows: for each cross-vector $v \in C_\gamma$, we set $\eta(v) := c(\alpha)$, for any $\alpha \in P_\gamma(v)$.
This value is well defined since $c$ is constant on $P_\gamma(v)$, as noted above.
In addition, it is clear that $\eta(-v) = \eta(v)$.
This shows that $\eta$ is a cross-functional.

To finish, let us show that $c_\eta = c $.
Given a path $\alpha \in P_\gamma$, we decompose it as a concatenation of smooth paths $\alpha_i \in P_\gamma$, where each $\alpha_i$ meets $\gamma$ once with certain velocity $v_i$, or not at all, in which case we say that $i$ is trivial.
Then we have
\[ c(\alpha)
= \sum_i c( \alpha_i )
= \sum_{i \text{ nontrivial}} c( \alpha_i )
= \sum_{i \text{ nontrivial}} \eta( v_i )
= c_\eta(\alpha). \qedhere\]
\end{proof}

\begin{remark}\label{R:Unitary} A cross-functional $\eta$ is a coorientation if and only if the cross-cochain $c_\eta$ is \term{unitary}, that is, it satisfies for each path $\alpha \in P_\gamma$ the condition
\[ c_\eta(\alpha) = \pm 1
\quad\text{ if }\quad
\Len_\gamma(\alpha)=1,\]
or, equivalently, the conditions
\[ c_\eta(\alpha) \leq \Len_\gamma(\alpha)
\quad\text{ and }\quad
c_\eta(\alpha) \equiv \Len_\gamma(\alpha) \mod 2.\]
\end{remark}

\subsection{Eulerian coorientations}

\begin{defi}\label{D:Eulerian}
A coorientation $\eta$ of~$\gamma$ is \term{Eulerian} if the cochain $c_\eta$ is closed, i.e. if $c_\eta(\alpha) = 0$ whenever $\alpha$ is a contractible closed curve, or, equivalently, if $c_\eta(\alpha)$ depends only on the homotopy class $\{\alpha\}$.
(The homotopies we consider here are with fixed endpoints and disregarding $\gamma$, meaning that the intermediate paths may not be in $P_\gamma$.)
The set of all Eulerian coorientations of~$\gamma$ is denoted by~$\Eul(\gamma)$.
\end{defi}

\noindent\begin{minipage}[t]{0.65\textwidth}
\hspace{3mm}
Equivalently, $\eta$ is Eulerian if around each double point $p$ of $\gamma$, among the four pieces of $\eta$ that meet at $p$ there are exactly two with positive coorientation and two with negative coorientation.
Hence the local picture of $\eta$ at $p$ is one of the following two:
either when travelling straight along~$\gamma$ and encountering~$p$ the coorientation changes ---in this case the coorientation is said to be \term{alternating} at~$p$---, or the coorientation does not change when following~$\gamma$ ---in which case it is~\term{non-alternating} at~$p$.
\end{minipage}
\begin{minipage}[t]{0.35\textwidth}
\begin{picture}(0,0)
\put(3,-38){\includegraphics[width=14mm]{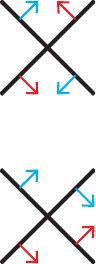}}
\put(22,-7){alternating}
\put(20,-32){non-alternating}
\end{picture}
\end{minipage}
\vspace{3mm}

\noindent
\begin{minipage}[t]{0.75\textwidth}
\begin{example}\label{Ex:Birkhoff}
If $[\gamma]_2\in\H^1(\Sigma;\ZZ/2\ZZ)$ is zero, then the regions of~$\Sigma\setminus\gamma$ can be colored in black and white in such a way that adjacent regions have different colors.
In this case we can coorient all edges toward the white regions.
The obtained coorientation is Eulerian, all double points being alternating.
\end{example}
\end{minipage}
\begin{minipage}[t]{0.25\textwidth}
  \begin{picture}(0,0)(0,0)
  \put(5,-24){\includegraphics[width=25mm]{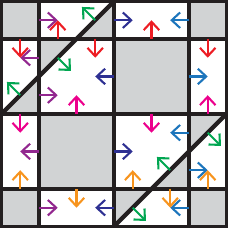}}
  \end{picture}
\end{minipage}

\vspace{6mm}
\noindent
\begin{minipage}[t]{0.25\textwidth}
  \begin{picture}(0,0)(0,0)
  \put(0,-24){\includegraphics[width=25mm]{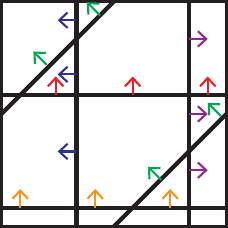}}
  \end{picture}
\end{minipage}
\begin{minipage}[t]{0.75\textwidth}
  \begin{example}\label{Ex:Brunella}
  There always exist global Eulerian coorientations, even when $[\gamma]_2\in\H_1(\Sigma; \ZZ/2\ZZ)$ is not zero.
  Indeed one can choose a coorientation for every component of~$\gamma$. 
  This yields an Eulerian coorientation having only non-alternating vertices. 
  If~$\gamma$ consists of~$c$ immersed curves, there are~$2^c$ such coorientations.
  \end{example}
\end{minipage}
\vspace{3mm}

\begin{remark}\label{L:EulerianHomology}
If $\eta$ is an Eulerian coorientation of~$\gamma$, then for every generic closed 1-chain~$\alpha$, the number~$c_\eta(\alpha)$ depends only of the homology class~$[\alpha]\in\H_1(\Sigma; \ZZ)$.
In consequence, $\eta$ induces a cohomology class~$[\eta]:= [c_\eta]$ in~$\H^1(\Sigma; \ZZ)$.
\end{remark}

We denote~$[\Eul(\gamma)]$ the subset of~$\H^1(\Sigma; \ZZ)$ consisting of the cohomology classes of the Eulerian coorientations on~$\gamma$.
Theorem~\ref{T:Coor} states that
\[ [\Eul(\gamma)] = \left\{ h \in \overline{B^*_{\xx}} \mid h_{\mmod 2} = [\gamma]_2 \right\}. \]
Let us prove the easy inclusion $\subseteq$, that is, that the cohomology class induced by any Eulerian coorientation is in the dual unit ball $B^*_{\xx}$ (i.e. it is $\leq \xx$) and also coincides with $[\gamma]_2$ modulo 2.

\begin{lemma}\label{L:Inclusion}
For every Eulerian coorientation $\eta$ of~$\gamma$ and every homology class~$a$ in~$\H_1(\Sigma; \ZZ)$, we have~$[\eta](a)\le\xxx{a}$ and also $[\eta](a) \equiv [\gamma]_2(a) \mod 2$.
\end{lemma}

\begin{proof}
Let $\alpha$ be an $\xx$-realizing curve of class~$a$.
Then $\langle \eta,\alpha\rangle$ counts every intersection point of~$\alpha$ and~$\gamma$ with a coefficient~$\pm1$, while~$\xxx{a}$ counts these same intersection points with a coefficient~$+1$ each.
Hence we have
\[ c_\eta(\alpha) \leq \xx(\alpha)
\quad \text{ and also } \quad
  c_\eta(\alpha) \equiv \xx(\alpha) \mod 2. 
\qedhere\]
\end{proof}

To prove the reverse inclusion we will use {eikonal functions}.

\subsection{Eikonal functions on the universal cover}
\label{S:Eikonal}
As before, $\Sigma$ is a compact closed surface with a multicurve $\gamma$ on it.

Our task now is to define the eikonal functions on the universal cover $\left( \wt\Sigma,\pi \right)$.
The space $\wt\Sigma$ has a multicurve $\wt\gamma := \pi^*(\gamma)$ (that is, the pullback of $\gamma$ by the map $\pi$), which induces a length functional $\Len_{\wt\gamma}$ and therefore, a distance function $d_{\wt\gamma}$, which we need to define the notion of eikonal functions.
However, we will instead define the distance function directly in terms of the multicurve $\gamma$, by taking advantage of the standard explicit construction of the universal cover.

We construct the universal cover $\left( \wt\Sigma, \pi \right)$ of the surface $\Sigma$ as follows.
The space~$\wt\Sigma$ is the set of homotopy classes of paths in $\Sigma$ starting at $p_0$, where $p_0 \in \Sigma \setminus \gamma$ is a fixed, arbitrary point.
Thus each point $x \in \wt\Sigma$ is of the form~$x = \{\alpha\}$ where $\alpha$ is a path in $\Sigma$ starting at $p_0$, and $\{\alpha\}$ denotes its homotopy class (with fixed endpoints).
In particular, the space $\wt\Sigma$ has a natural base point $x_0 = \{1_{p_0}\}$, where $1_{p_0}$ is the trivial path at $p_0$.
The covering map $\pi:\wt\Sigma \to \Sigma$ is the function that sends each homotopy class $\{\alpha\}$ to the endpoint of the path $\alpha$.

The fundamental group $\Pi_1(\Sigma,p_0)$, hereafter denoted $\Pi_1$, acts (on the left) on~$\wt\Sigma$ as follows: each loop homotopy class $\{\beta\} \in \Pi_1$ induces on $\wt\Sigma$ a transformation $T_\beta: \{\alpha\} \mapsto \{\beta\alpha\}$, where $\beta\alpha$ is the concatenation of the path $\beta$ followed by the path~$\alpha$.
This action commutes with the covering map $\pi$ (that is, it satisfies $\pi \circ T_\beta = \pi$ for all $\{\beta\} \in \Pi_1$) and is transitive on each fiber of $\pi$.

The length with respect to $\gamma$ of a homotopy class $\{\alpha\}$ is defined as the minimum length of a generic path $\alpha'$ in the class,
\[ \Len_\gamma\{ \alpha \} =  \min_{\alpha'\in\{\alpha\}\cap P_\gamma} \Len_\gamma( \alpha' ). \]
The \term{distance} between two points $x=\{\alpha\}$,~$y =\{\beta\} \in \wt\Sigma$ not located on the multicurve $\wt\gamma := \pi^*(\gamma)$ is
\[ d_{\wt\gamma}( x, y ) = \Len_{\gamma}\{ \alpha^\dag \beta \} \]
where $\alpha^\dag$ is the reverse of the path $\alpha$.
Note that $d_{\wt\gamma}$ satisfies the triangle inequality, therefore it is an integer-valued (but non positive-definite) distance function on $\wt\Sigma \setminus \wt\gamma$.
Moreover, an easy computation shows that the transformations $T_\beta$ preserve this distance function.

\begin{defi}\label{def:eikonal}
An integer-valued function~$f$ defined on a subset $D$ of $\wt\Sigma \setminus \wt\gamma$ is said \term{pre-eikonal} if it satisfies
\begin{equation}\label{eq:pre-eikonal}
| f(y) - f(x) | \leq d_{\wt\gamma}( y, x ) \quad\text{ and }\quad
f(y) - f(x) \stackrel{\mmod 2}{\equiv} d_{\wt\gamma}( y, x )
\quad\text{ for all $x$,~$y \in D$},
\end{equation}
An \term{eikonal function} is a pre-eikonal function defined on the whole set $\wt\Sigma \setminus \wt\gamma$.
\end{defi}

\begin{remark}
A function $f:\wt\Sigma \setminus \wt\gamma \to \ZZ$ is eikonal if and only if it satisfies the local condition
\[ f(y) - f(x) = \begin{cases}
0     &\text{ when }d_{\wt\gamma}(x,y) = 0\\
\pm 1 &\text{ when } d_{\wt\gamma}(x, y) = 1.\end{cases} \]
\end{remark}

The term ``eikonal function'' comes from geometric optics, where it describes a (possible singular) real-valued function $f$ that solves the eikonal equation $\|\nabla f\| \equiv 1$.
The eikonal functions defined above are discrete analogues of these real-valued functions.

\begin{defi}
An integer-valued function~$f$ defined on a subset $D$ of $\wt\Sigma \setminus \wt\gamma$ is said \term{equivariant} with respect to a cohomology class $h \in \H^1(\Sigma; \ZZ)$, or $h$-equivariant, if it satisfies
\[ f(y) - f(x) = h[\beta] \]
for all pairs of points $x, y \in D$ and all loop homotopy classes $\{\beta\} \in \Pi_1$ such that $T_\beta(x) = y$.
\end{defi}

Every Eulerian coorientation $\eta$ of $\gamma$ determines a function $f_\eta: \wt\Sigma \setminus \wt\gamma \to \ZZ$, called the \term{primitive} of $\eta$, by the formula
\[ f_\eta: \{\alpha\} \mapsto c_\eta(\alpha). \]
The number $c_\eta(\alpha)$ does not depend on how the path $\alpha$ is chosen within its homotopy class since $\eta$ is Eulerian.

\begin{lemma}\label{L:bijection_eulerian_eikonal}
For each cohomology class $h \in \H^1(\Sigma;\ZZ)$, the map $\eta \mapsto f_\eta$ bijects the set of Eulerian coorientations of~$\gamma$ of cohomology class $h$ to the set of $h$-equivariant eikonal functions on $\wt\Sigma \setminus \wt\gamma$ that vanish at the base point $x_0 = \{1_{p_0}\}$.
\end{lemma}

\begin{proof}
Let us first see that for each Eulerian coorientation $\eta$, the function $f_\eta$ is eikonal, $[\eta]$-equivariant, and vanishes at the base point.
The last claim is clear: since the trivial path $1_{p_0}$ does not meet $\gamma$, we have
\[ f_\eta(x_0) = c_\eta(1_{p_0}) = 0. \]
To see that $f$ is eikonal, take two points $\{\alpha\}$,~$\{\beta\} \in \wt\Sigma\setminus \wt\gamma$ at distance $d_{\wt\gamma}\left( \{\alpha\},\{\beta\} \right) = 1$.
This means that there exists a path $\eps \in P_\gamma$ homotopic to $\alpha^\dag\beta$ with $\Len_\gamma(\eps) = 1$.
Hence we can verify that
\begin{align*}
f_\eta\{\beta\} - f_\eta\{\alpha\}
=  c_\eta(\beta) - c_\eta(\alpha)
&= c_\eta(\alpha^\dag \beta)
  &\\
&= c_\eta(\eps)
  &\text{since $\eta$ is Eulerian}\\
&= \pm 1.
  &
\end{align*}
Similarly, one can see that $f_\eta\{\beta\} = f_\eta\{\alpha\}$
if $d_{\wt\gamma}\big(\{\alpha\},\{\beta\} \big) = 0$.
Finally, to see that $f_\eta$ is $[\eta]$-equivariant, we take a loop homotopy class $\{\beta\} \in \Pi_1$ and a point $\{\alpha\} \in \wt\Sigma \setminus \wt\gamma$ and we verify that
\begin{align*}
f_\eta\{\beta \alpha\}
= c_\eta(\beta \alpha)
&= c_\eta(\beta) + c_\eta(\alpha) \\
&= [\eta] [\beta] + f_\eta\{\alpha\}.
\end{align*}

Now let us fix a cohomology class $h \in \H^1(\Sigma;\ZZ)$.
As we have just shown, the map $\eta \mapsto f_\eta$ restricts to a map $R_h$ from the set of Eulerian coorientations of class $h$ to the set of $h$-equivariant eikonal functions on $\wt\Sigma\setminus \wt\gamma$ that vanish at the base point $x_0=\{1_{p_0}\}$.
Let us show that $R_h$ is bijective.

To prove that $R_h$ is injective, fix an Eulerian coorientation $\eta$ and a cross-vector $v \in C_\gamma$.
We shall express $\eta(v)$ in terms of the function $f_\eta$.
Take a path $\eps \in P_\gamma$ which crosses $\gamma$ just once with velocity $v$, and let $\alpha \in P_\gamma$ be an auxiliary path from $p_0$ to the startpoint of $\eps$.
Then we have
\[ f_\eta\{\alpha \eps\} - f\{\alpha\}
= c_\eta(\alpha\eps) - c_\eta(\alpha)
= c_\eta(\eps) = \eta( v ), \]
which shows that $\eta$ can be recovered from the function $f_\eta$, and thus $R_h$ is injective.

Finally, let us show that $R_h$ is surjective.
Let $f:\wt\Sigma \setminus \wt\gamma \to \ZZ$ be an equivariant eikonal function that vanishes at the base point $x_0 = \{1_{p_0}\}$.
We have to construct an Eulerian coorientation $\eta$ such that $f_\eta = f$.
To do so, we define first a cross-cochain $c$ as follows.
For any generic smooth path $\eps \in P_\gamma$, we let
\[ c(\eps) := f\{\alpha \eps\} - f\{ \alpha \} \]
where $\alpha \in P_\gamma$ is an auxiliary path from $p_0$ to the startpoint of $\eps$.
Let us show first that the value $c(\eps)$ is well defined.
Let $\alpha'$ be any other path from $p_0$ to the startpoint of $\eps$.
Then we can write $\{ \alpha '\} = \{\beta \alpha\}$ where $\beta := \alpha^\dag \alpha'$, and thus from the fact that $f$ is $h$-equivariant for some $h:\Pi_1 \to \ZZ$ we get
\begin{align*}
f\{\alpha' \eps\} - f\{\alpha'\}
&= f\{\beta \alpha \eps\} - f\{\beta \alpha\}
  & \\
&= h[\beta] + f\{\alpha \eps\}
- h[\beta] - f\{ \alpha \}
  &\text{since $f$ is $h$-equivariant}\\
&= f\{\alpha \eps\} - f\{ \alpha \}.
  &
\end{align*}

We claim that $c$ is a cross-cochain according to Definition~\ref{D:cross-cochain}, and in fact, a unitary and closed cross-cochain.

We note first that $c(\eps)$ only depends on the homotopy class $\{\eps\}$.
(This will imply that $c$ is closed as a cross-cochain.)

Let us prove that $c$ is additive with respect to concatenation of paths.
Let $\delta, \eps \in P_\gamma$ be consecutive paths.
To show that $c(\delta\eps) = c(\delta) + c(\eps)$, we take an auxiliary path $\alpha \in P_\gamma$ from $p_0$ to the startpoint of $\delta$.
Then we have
\begin{align*}
c( \delta \eps )
&= f\{ \alpha \delta \eps \} - f\{ \alpha \} \\
&= f\{ \alpha \delta \eps \} - f\{ \alpha \delta \}
  + f\{ \alpha \delta \} - f\{ \alpha \} \\
&= c( \eps ) + c( \delta ).
\end{align*}
Similary, let us show that $c$ is alternating with respect to path reversion.
Consider a path $\eps \in P_\gamma$ and its reverse $\eps^\dag$, and let $\alpha \in P_\gamma$ be an auxiliary path from $p_0$ to the startpoint of $\eps$.
Note that the path $\alpha\eps$ goes from $p_0$ to the startpoint of $\eps^\dag$, therefore we have
\begin{align*}
c( \eps^\dag )
&= f\{ \alpha \eps \eps^\dag \} - f\{ \alpha \eps \} \\
&= f\{ \alpha \} - f\{ \alpha \eps \} \\
&= - c( \eps ).
\end{align*}

Next, let us show that $c$ is supported on $\gamma$, i.e. that $c(\eps) = 0$ for any path $\eps \in P_\gamma$ that avoids $\gamma$.
Let $\eps \in P_\gamma$ be such a path, and let $\alpha \in P_\gamma$ be an auxiliary path from $p_0$ to the startpoint of $\eps$.
Then we have
\begin{align*}
c( \eps )
&= f\{ \alpha \eps \} - f\{ \alpha \}
  &\\
&\leq d_{\wt\gamma}\big( \{ \alpha \}, \{ \alpha \eps \} \big)
  &\text{since $f$ is an eikonal function}\\
&= \Len_\gamma\{ \alpha^\dag \alpha \eps \}
= \Len_\gamma\{ \eps \}
= 0.
\end{align*}
Similarly, let us show that $c$ is unitary.
For a path $\eps \in P_\gamma$ of length $\Len_\gamma(\eps) = 1$, we have to show that $c(\eps) = \pm 1$.
The fact that $\Len_\gamma(\eps) = 1$ implies that $\Len_\gamma\{\eps\} = 1$,
since the possibility $\Len_\gamma\{\eps\} = 0$ is excluded because homotopic paths have the same length modulo 2.
To compute $c(\eps)$ we take an auxiliary path $\alpha \in P_\gamma$ from $p_0$ to the startpoint of $\eps$ and we note that
\[ c( \eps ) = f\{ \alpha \eps \} - f\{ \alpha \} = \pm 1 \]
since $f$ is an eikonal function and
\[ d_{\wt\gamma} \big( \{\alpha\}, \{\alpha \eps\} \big)
= \Len_\gamma\{ \alpha^\dag \alpha \eps \}
= \Len_\gamma\{ \eps \} = 1. \]

Finally, let us show that $c$ is constant on any continuous family $(\eps_t)_{t\in[0,1]}$ of smooth paths $\eps_t \in P_\gamma$.
It suffices to verify that $c(\eps_0) = c(\eps_1)$.
Denote $r_t$ and $s_t$ the startpoint and endpoint of $\eps_t$ for each $t\in[0,1]$.
Note that the curves $r:t \mapsto r_t$ and $s:t \mapsto s_t$ avoid the multicurve $\gamma$.
These curves $r,s$ may not be in $P_\gamma$, but they surely can be approximated by respective curves $\rho, \sigma \in P_\gamma$ that are homotopic to $r$ and $s$ respectively (with fixed endpoints) and also avoid $\gamma$.
Then we have $\{\eps_0\} = \{\rho \eps_1 \sigma^\dag \}$, which
 implies that
\[ c(\eps_0) = c(\rho) + c(\eps_1) - c(\sigma) = c(\eps_1) \]
since $c(\rho) = c(\sigma) = 0$ because $\rho$ and $\sigma$ avoid $\gamma$.

This finishes the proof that $c$ is a closed, unitary cross-cochain.
Therefore, by Proposition~\ref{L:bijection_functional_cochain} (together with Remark~\ref{R:Unitary}), there exists an Eulerian coorientation $\eta$ of $\gamma$ such that $c = c_\eta$.
We see that $f_\eta = f$ because for any homotopy class $\{\alpha\} \in \wt \Sigma \setminus \wt\gamma$ (represented by a generic smooth path $\alpha \in P_\gamma$ starting at $p_0$ ) we have
\[ f_\eta\{\alpha\}
= c_\eta(\alpha)
= c(\alpha )
= f\{1_{p_0} \alpha\} - f\{1_{p_0}\}
= f\{\alpha\} \]
since $f\{1_{p_0}\} = 0$.
This shows that $f_\eta = f$, concluding the proof that $R_h$ is surjective.
\end{proof}

The next result is the key to proving Theorem~\ref{T:Coor}.

\begin{lemma}[Extension]\label{L:extension}
Every pre-eikonal function $f$ defined on a subset $D$ of $\wt \Sigma \setminus \wt \gamma$ can be extended to an eikonal function $\ov f:\wt \Sigma \setminus \wt\gamma \to \ZZ $ given by
\[ \ov f(x) = \min_{y\in D} f(y) + d_{\wt \gamma}(x,y). \]
\end{lemma}

\begin{proof}
For a point $x \in \wt\Sigma \setminus \wt\gamma$, we want to define~$\ov f(x)$.
We first observe that, for every $y\in D$, the value~$\ov f(x)$ must lie in the interval $I_{x,y} :=
\left[ f(y) - d_{\wt \gamma}(x,y), f(y) + d_{\wt \gamma}(x,y) \right]$.
See Figure~\ref{F:Eikonal}.

\begin{figure}
\centering
\begin{picture}(70,50)
\put(0,0){\includegraphics[width=70mm]{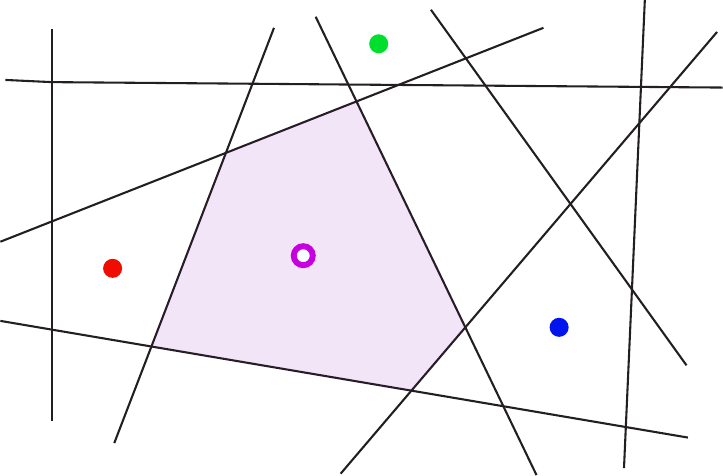}}
\put(27,24){$x$}
\put(7,17){$y_1$}
\put(-17,23){$f(y_1)=0$}
\put(-20,17){$I_{x,y_1}=[-1,1]$}
\put(38,41){$y_2$}
\put(16,46){$f(y_2)=2 \quad I_{x,y_2}=[-1,5]$}
\put(49,14){$y_3$}	
\put(64,17){$f(y_3)=-1$}
\put(62,8){$I_{x,y_3}=[-3,1]$}
\end{picture}
\caption{\small A part of the multicurve~$\wt\gamma$ (black and thin).
Assume that the set~$D$ consists of three points~$y_1, y_2, y_3$ (red, green and blue dots) with prescribed values $f(y_1)=0, f(y_2)=2$ and $f(y_3)=-1$.
Considering a fourth point~$x$ (purple), we see that we have
$I_{x,y_1}=[-1,5]$,
$I_{x,y_2}=[-1,1]$ and
$I_{x,y_3}=[-3,1]$.
In particular these three intervals intersect, and one can set~$\ov f(x) = 1$. }
\label{F:Eikonal}
\end{figure}
We claim that for every $y, y'$ in~$D$, the intervals $I_{x, y}, I_{x, y'}$ intersect.
Otherwise there would exist two points $y, y'$ such that $f(y)+d_{\wt \gamma}(x,y)<f(y')-d_{\wt \gamma}(x,y')$, which implies $f(y')-f(y)>d_{\wt \gamma}(x,y)+d_{\wt \gamma}(x,z)\ge d_{\wt \gamma}(y,y')$, contradicting pre-eikonality of $f$.
Now, any set of intervals in~$\RR$ that pairwise intersect has a global common point.
Therefore the intersection $\displaystyle\cap_{y\in D} I_{x,y}$ is non-empty.
So we define $\ov f(x)$ as the highest common point $\ov f(x):=\min_{y\in D} f(y)+d_{\wt \gamma}(x,y)$ of these intervals.

We claim that the extension~$\ov f$ is pre-eikonal (and therefore eikonal, since it is defined at all points of $\wt\Sigma \setminus \wt\gamma$).
Indeed, to prove that $|f(x')-f(x)|\leq d_{\wt \gamma}(x,x')$, it is enough to check that
\[ \left| \left( f(y) + d_{\wt \gamma}(x',y) \right)
         -\left( f(y) + d_{\wt \gamma}(x ,y) \right) \right|
\leq d_{\wt \gamma}(x,x')\]
for each $y$, which follows from the triangle inequality in the form
\[ \left| d_{\wt \gamma}(x',y)
        - d_{\wt \gamma}(x,y) \right| \leq d_{\wt \gamma}(x,x'). \]
To prove that $ f(x') - f(x) \equiv d_{\wt \gamma}(x,x') $ modulo 2, we write
\begin{align*}f(&x')-f(x)=\\
&= \left( f(y') + d_{\wt \gamma}(x',y') \right)
 - \left( f(y ) + d_{\wt \gamma}(x ,y ) \right)
\quad\text{ for certain $y,y'\in D$}\\
&\equiv d_{\wt \gamma}(y,y')+d_{\wt \gamma}(x',y')-d_{\wt \gamma}(x,y)
\quad\text{ modulo 2 since $f$ is pre-eikonal}\\
&\equiv d_{\wt \gamma}(y,y')+d_{\wt \gamma}(x',y')+d_{\wt \gamma}(x,y)
\quad\text{ since plus and minus coincide mod 2}\\
&\equiv d_{\wt \gamma}(x,x')
\quad\text{ since homotopic paths have equal length mod 2.\qedhere}
\end{align*}
\end{proof}

Note that a pre-eikonal function $f$ generally admits several eikonal extensions.
The one we denoted $\ov f$ is the \emph{highest} one.
It has the advantage of being determined by $f$ by an explicit formula.

\subsection{Proof of Theorem~\ref{T:Coor}}
As explained after Lemma~\ref{L:Inclusion}, it remains to be shown that every cohomology class $h \in \BBx$ that coincides modulo 2 with~$[\gamma]_2$ is the cohomology class of some Eulerian coorientation $\eta$.
We fix such a cohomology class $h$.

Recall that we have chosen a point $p_0$ in~$\Sigma \setminus \gamma$ to construct the universal cover~$\wt\Sigma$ and the covering map $\pi:\wt\Sigma \to \Sigma$.
Denote $D = \pi^{-1}(p_0)$.
We define a function $f:D \to \ZZ$ by the formula $f\{\alpha\} = h [\alpha] $.
This function is well defined because homotopic paths are homologous.

\begin{claim}
The function $f:D \to \ZZ$ is an $h$-equivariant pre-eikonal function.
\end{claim}

\begin{proof}
Let us show that $f$ is $h$-equivariant.
Take a loop $\beta$ in $\Sigma$ based at the point~$p_0$.
Then for points $y = \{\alpha\}$,~$y'=T_\beta(y) = \{\beta \alpha\} \in D$ we have
\[ f(y') - f(y)
= h[ \beta \alpha ] - h[ \alpha ]
= h[ \beta ], \]
as claimed.
To show that $f$ is pre-eikonal we continue as follows.
Any two points $y$, $y'\in D$ can be written as $y = \{\alpha\}$, $y' = T_\beta(y) = \{\beta\cdot\alpha\}$.
Therefore we have
\[ f(y') - f(y) = h[\beta]
\leq \xx[\beta] \]
since $h \in \BBx$.
On the other hand, the distance between $y$ and $y'$ is
\begin{align*}
d_{\wt\gamma}(y, y')
&= \Len_\gamma\{\alpha^\dag \beta\alpha\}
  &\\
&= \Len_\gamma(\beta')
  &\text{for some path $\beta' \in \{\alpha^\dag \beta\alpha\}$,}\\
&\geq \xx[\beta']
  &\text{by definition of $\xx$,}\\
&= \xx[\beta]
  &\text{since $[\beta']=[\alpha^\dag \beta\alpha]=[\beta]$,}
\end{align*}
which shows that $f(y') - f(y) \leq d_{\wt\gamma}(y,y')$.
To see that $f(y') - f(y) \equiv d_{\wt\gamma}(y,y')$ modulo 2 we note that
\begin{align*}
f(y') - f(y)
&= h[\beta] \\
& \equiv [\gamma]_2[\beta]
  &\text{since $h \equiv [\gamma]_2$ modulo 2 } \\
& = [\gamma]_2[\beta']
  &\text{ since $[\beta'] = [\beta]$, with $\beta'$ as above}\\
& \equiv \Len_\gamma( \beta ' )
  &\text{modulo 2 by definition of $[\gamma]_2$}\\
& = d_{\wt\gamma}( y, y' ).
\end{align*}
This finishes the proof that $f$ is a pre-eikonal function.
\end{proof}

By the Extension Lemma~\ref{L:extension}, we can extend $f$ to an eikonal function $\ov{f}: \wt\Sigma \setminus \wt\gamma \to \ZZ$ defined by the formula
$\ov{f}(x) = \min_{y\in D} f(y) + d_{\wt \gamma}( y , x) $.

\begin{claim}
The function $\ov{f}$ is $h$-equivariant.
\end{claim}
\begin{proof} This follows from the fact that $f$ is $h$-equivariant.
Indeed, take a loop homotopy class $\{\beta\}$ in $\Pi_1$ and a point $x \in \wt\Sigma \setminus \wt\gamma$.
Then the value of $f$ at the translate point~$x' = T_\beta(x)$ is
\begin{align*}
\ov{f}(x')
&= \min_{y' \in D} f( y' ) + d_{\wt\gamma}( y', x' )
  \\
&= \min_{y \in D} f( T_\beta(y) ) + d_{\wt\gamma}( T_\beta(y), T_\beta(x) )
  \qquad\text{since $D = T_\beta(D)$}\\
&= \min_{y \in D} f( y ) + h[\beta] + d_{\wt\gamma}( y, x )
  \quad\text{since $f$ is $h$-equivariant and $T_\beta$ preserves $d_{\wt\gamma}$}\\
&= \ov f(x) + h[\beta].
  \qedhere
\end{align*}
\end{proof}

Since $f$ is an $h$-equivariant eikonal function, by Lemma~\ref{L:bijection_eulerian_eikonal} there exists a unique Eulerian coorientation $\eta$ with cohomology class $[\eta] = h$ such that $f_\eta = \ov f$.

Let us put everything together.
We have shown in Theorem~\ref{T:NORM} that $\xx$ is an integral seminorm on $\H_1(\Sigma; \ZZ)$, and this seminorm coincides modulo 2 with the cohomology class $[\gamma]_2$.
Therefore we can apply Theorem~\ref{T:Thurston} (the extension of Thurston's theorem).
We conclude that for each homology class $a \in \H_1(\Sigma; \ZZ)$, we have
\[ \xx(a)
= \max_{\substack{\varphi \in B^*_{\xx}\\h_{\mmod 2} = [\gamma]_2}} h(a)
= \max_{\eta \in \Eul(\gamma)} [\eta](a). \]
This concludes the proof of Theorem~\ref{T:Coor}.

\section{Birkhoff sections with symmetric boundary for the geodesic flow}\label{S:Birkhoff}

We now turn to geodesic flows on unit tangent bundles to hyperbolic surfaces and their Birkhoff sections. 
Unlike the two previous sections, the surfaces we consider are now equipped with a hyperbolic metric, and all considered multicurves are geodesic. 
We first recall in~\ref{S:Geodesic} what are the geodesic flow and the symmetric lift of a geodesic.
Then in~\ref{S:Construction} we associate to every Eulerian coorientation a surface in the unit tangent bundle needed for proving the first part of Proposition~\ref{T:Brunella}.
We recall in~\ref{S:SSF} the basic definitions on Birkhoff sections and the elements of Schwartzman--Fried--Sullivan Theory we need for our classification.
Then in~\ref{S:Anosov} we recall basic notions on pseudo-Anosov flows and Fried's result on their homology directions.
In~\ref{S:Affine} we make a bit of elementary algebraic topology for describing homology classes of surfaces with boundary.
This allows to prove in~\ref{S:Proof} the second part of Proposition~\ref{T:Brunella}, as well as Theorem~\ref{T:Classification}.

\subsection{Geodesic flow and symmetric collections of orbits}
\label{S:Geodesic}

Given a hyperbolic surface~$\Sigma$, its \term{unit tangent bundle} is the circle bundle~$\U\Sigma$ made of length 1 tangent vectors, that is $\U\Sigma=\{(p, v)\in \mathrm{T}\Sigma\,\big|\, \| v\|=1\}$.
The \term{geodesic flow}~$\fgeodt$ on~$\U\Sigma$ is the flow whose orbits are lifts of geodesics.
Namely for~$\alpha$ a geodesic on~$\Sigma$ parametrized with speed one, the orbit of~$\fgeodt$ going through the point~$(\alpha(0), \dot \alpha(0))\in\U\Sigma$ is described by $\fgeod^t((\alpha(0), \dot \alpha(0))=(\alpha(t), \dot \alpha(t))$.
For every oriented periodic geodesic~$\abas$ on~$\Sigma$, there is one periodic orbit of~$\fgeodt$ corresponding to the oriented lift of~$\abas$ and denoted by~$\vec \alpha$.
If $\alpha$ now denotes an unoriented geodesic on~$\Sigma$, there are two associated periodic orbits of~$\fgeodt$, one for each orientation.
We denote by~$\syma$ the union of these two periodic orbits, it is an oriented link in~$\U\Sigma$ that is invariant under the involution~$(p,v)\mapsto(p, -v)$.
A link of the form~$\syma_1\cup\dots\cup\syma_k$ is called a \term{symmetric link}.\footnote{The term ``antithetic link'' was suggested by Bruce Bartlett, but we remarked that symmetric is already used in the literature.}.

\subsection{Birkhoff--Brunella surfaces and the first part of Proposition~\ref{T:Brunella}}
\label{S:Construction}

Starting from a hyperbolic surface~$\Sigma$ and a finite collection~$\gamma$ of periodic geodesics\footnote{
In the sequel we always assume~$\gamma$ to be in general position, meaning in particular that no point belong to three different arcs. 
This is a restriction as there exists collection of geodesics on surfaces that exhibit triple points for all constantly curved metrics. 
One way to deal with this situation is to perturb the metric, allowing the curvature to slighty change so that the position of the collection becomes general. 
Indeed the arguments we use do not require constant curvature, only negative.}
on~$\Sigma$, we now explain how to associate to every Eulerian coorientation of~$\gamma$ a surface in~$\U\Sigma$ bounded by~$-\vecgamma$ and transverse to~$\fgeodt$, thus proving the first part of Proposition~\ref{T:Brunella}.

Fix a coorientation~$\eta$ (not yet Eulerian) of~$\gamma$.
For every edge $e$ of~$\gamma$ (\emph{i.e.} segment between two double points), we consider the set~$R^{e, \eta}$ of those tangent vectors based on~$e$ and pairing positively with~$\eta$.
It is a subset of in~$\U\Sigma$ of the form~$e\times[0, \pi]$ (see Figure~\ref{F:Segment}), hence we call it an \term{elementary rectangle}. 
With the notation of Section~\ref{S:Coor}, it is the closure of a connected component of~$C_\gamma$. 
Is is bounded by the two lifts of~$e$ in~$\U\Sigma$ (called the \term{horizontal part} of~$\partial R^{e, \eta}$) and two halves of the fibers of the extremities of~$e$ (called the \term{vertical part} of $\partial R^{e, \eta}$).
Note that the interior of~$R^{e, \eta}$ is transverse to the geodesic flow~$\fgeodt$ while the horizontal part of~$\partial  R^{e, \eta}$ is tangent to it.
We then orient~$R^{e, \eta}$ so that orbits of~$\fgeodt$ intersect it positively.
One checks that the induced orientation on~$\partial  R^{e, \eta}$ is opposite to the one given by~$\fgeodt$, as explained in Figure~\ref{F:Segment}. 
This is the reason why we want to consider negative orientations in Theorem~\ref{T:Classification} and Proposition~\ref{T:Brunella}.

\begin{figure}
  \centering
  \includegraphics[width=.5\textwidth]{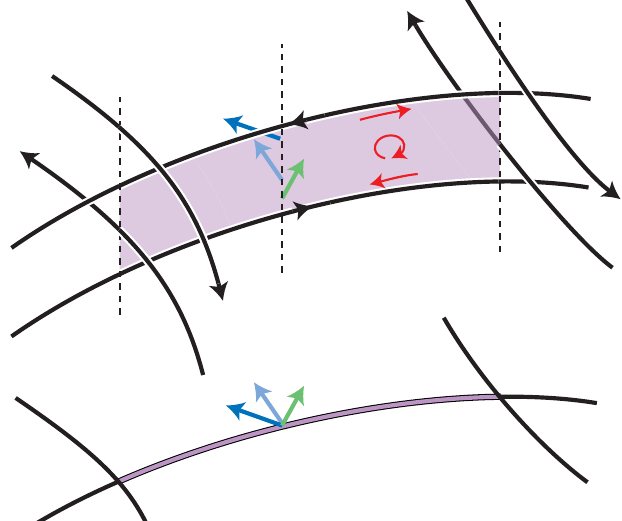}
  \caption{\small Bottom: an edge~$e$ of~$\gamma$ and a coorientation~$\eta$ on it. Top: the corresponding rectangle~$R^{e, \eta}$ in~$\U\Sigma$.
  The dotted lines represent the fibers of some points of~$\Sigma$, that is, each point on these lines represent a unit tangent vector to~$\Sigma$.
  Since the fibers are actually circles, the top and bottom extremities of the dotted lines should be glued.
  The rectangle~$R^{e, \eta}$ is transverse to the orbits of~$\fgeodt$ and the induced orientation is shown in red.
  The induced orientation of the horizontal boundary of~$R^{e, \eta}$ (red) is opposed to the orientation of the flow (black).
  Thus the surfaces we construct are transverse surfaces whose boundary components have multiplicity~$-1$. }
  \label{F:Segment}
\end{figure}

Consider now the 2-dimensional CW-complex~$\Cnu$ that is the union of the rectangles~$R^{e, \eta}$ over all edges~$e$ of~$\gamma$, see the left parts of Figures~\ref{F:BirSurface} and~\ref{F:BruSurface}.

\begin{lemma}{\label{L:Snu}}
The 2-complex~$\Cnu$ described above has boundary~$-\vecgamma$ if and only if the coorientation~$\eta$ is Eulerian.
\end{lemma}

\begin{proof}
Since~$\Cnu$ is the union of one rectangle per edge of~$\gamma$, the horizontal boundary of~$\Cnu$ is always in~$\vecgamma$.
Since the orientation is opposite to the geodesic flow (see Figure~\ref{F:Segment}), it is actually~$-\vecgamma$.

What we have to check is that the vertical boundary is empty if and only if~$\eta$ is Eulerian.
At every double point~$v$ of~$\gamma$ there are four incident rectangles, corresponding to the four adjacent edges.
Now the vertical boundary of a rectangle~$R^{e, \eta}$ is oriented upwards at the right extremity of~$e$ (when cooriented by~$\eta$) and downwards at the left extremity.
Then the vertical boundary in a vertex of~$\gamma$ is empty if only if all vertical contributions cancel.
This is the case exactly when two edges are cooriented in a direction, and two others in the opposite direction: this means that~$\eta$ is Eulerian around~$v$.
Conversely, if~$\eta$ is Eulerian, then up to rotation there are two local configurations around~$v$ (that we called alternating and non-alternating), and one checks that in both cases, the vertical boundary is empty (see the left parts of Figures~\ref{F:BirSurface} and~\ref{F:BruSurface}).
\end{proof}

\begin{figure}[htb]
  \centering
  \includegraphics[width=.6\textwidth]{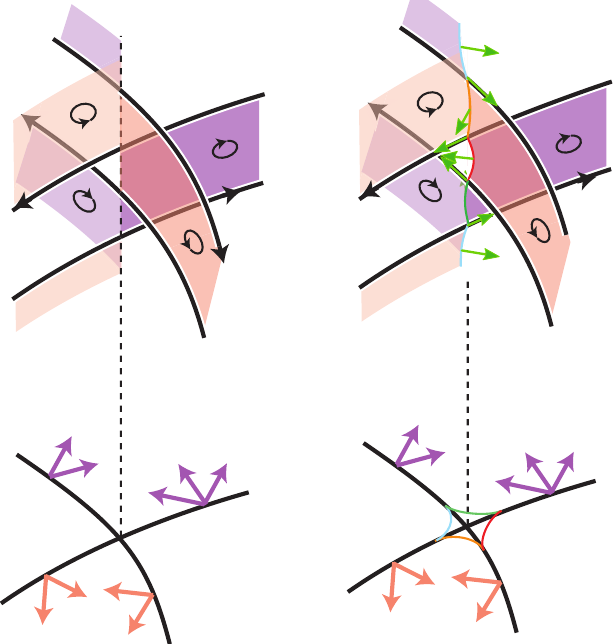}
  \caption{\small On the left, the complex~$\Cnu$ around the fiber of an alternating double point of~$\gamma$.
  Every point of the fiber of~$v$ is adjacent to exactly two rectangles.
  On the right the surface~$\Snu$ is obtained by smoothing~$\Cnu$ in a neighborhood of the fiber of the double point.
  Its interior is transverse to the vector field generating the geodesic flow~(green).}
  \label{F:BirSurface}
\end{figure}

\begin{figure}[htb]
  \centering
  \begin{picture}(90,70)
  \put(0,0){\includegraphics[width=87mm]{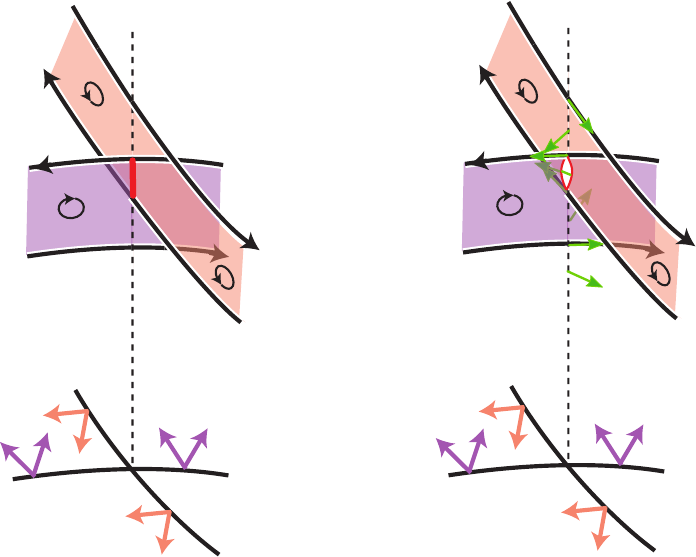}}
  \put(57,18){$1$}
  \put(78,19){$2$}
  \put(80,6){$3$}
  \put(64,6){$4$}
  \end{picture}
  \caption{\small On the left, the complex~$\Cnu$ around the fiber of a non-alternating double point of~$\gamma$.
  Every point of the fiber of~$v$ is adjacent to an even number of rectangles.
  On the right the surface~$\Snu$ is obtained by desingularizing~$\Cnu$ on the portion of the fiber where four rectangles meet.
  Note that the topology of the complex changes in this process.
  However its interior is still transverse to the vector field generating the geodesic flow (green).}
  \label{F:BruSurface}
\end{figure}

When~$\eta$ is Eulerian, the complex~$\Cnu$ is not a topological surface if~$\eta$ has some non-alternating points: as depicted on Figure~\ref{F:BruSurface}, there are edges in the vertical boundary of four adjacent rectangles, instead of two for obtaining a topological surface.
But it is the only obstruction and we can desingularize such segments as shown on the right of Figure~\ref{F:BruSurface}. 
More precisely, label by~$1, 2, 3, 4$ the quadrants around the considered non-alternating point so that two edges point toward~$1$ under the coorientation~$\eta$. 
Then the set~$s$ of those tangent vectors based on the double point and pointing toward quadrant number~$1$ is the singular segment to which four rectangles are adjacent. 
We thus split~$s$ into two segments~$s_{1}$ and~$s_3$, so that the extremities of both segments (in~$\U\Sigma$) coincide with the extremities of~$s$, but $s_1$ is pushed a bit into quadrant number~$1$, and $s_3$ is pushed a bit into quadrant number~$3$. 
Then we distort a bit the two rectangles adjacent to quadrant~$1$ so that their vertical boundary is~$s_1$, and we distort a bit the two rectangles adjacent to quadrant~$3$ so that their vertical boundary is~$s_3$. 
These gluings are made in a smooth way.

The main tool connecting Eulerian coorientations to Birkhoff sections is the following.

\begin{defi}\label{D:Snu}
For $\eta$ an Eulerian coorientation, the associated \term{BB-surface} is the surface~$\Snu$ obtained from~$\Cnu$ by desingularizing and smoothing the fibers of the double points of~$\gamma$, as on the right parts of Figures~\ref{F:BirSurface} and~\ref{F:BruSurface}.
\end{defi}

The term {BB} stands for Birkhoff--Brunella, as this construction generalises previous constructions by these two authors.
Indeed, the BB-surface associated to a Birkhoff coorientation (Example~\ref{Ex:Birkhoff}) is isotopic to the construction suggested by Birkhoff and popularized by Fried~\cite{birkhoff1917dynamical, fried1983transitive}.
Also the BB-surface associated to a Brunella coorientation (Example~\ref{Ex:Brunella}) was introduced by Brunella~\cite[Description~2]{brunella1994discrete}.
This construction already yields the first part of Proposition~\ref{T:Brunella}:

\begin{prop}\label{P:Transverse}
For $\Sigma$ a hyperbolic surface, $\gamma$ a geodesic multicurve, and $\eta$ an Eulerian coorientation of~$\gamma$, the associated surface~$\Snu$ is embedded in~$\U\Sigma$, it is bounded by~$-\symgamma$, and its interior is transverse to the orbits of the geodesic flow~$\fgeodt$.
\end{prop}

\begin{proof}
The surface $\Snu$ is obtained by desingularizing~$\Cnu$, so it is embedded.
Its boundary coincide with the boundary of~$\Cnu$, so it is (with orientation) $-\symgamma$.
Finally, the desingularization preserves the transversality to~$\fgeodt$.
Since~$\Cnu$ is positively transverse to~$\fgeodt$ away from its boundary, so is~$\Snu$.
\end{proof}

\subsection{Birkhoff sections and Schwartzman--Fried--Sullivan Theory}
\label{S:SSF}

Our goal here is to present a criterion for the existence of a Birkhoff section in a given homology class. 
Such a criterion exists when the Birkhoff section has no boundary (in this case we call it a global cross section), and it goes back to Schwartzman. 
It can be adapted to Birkhoff sections using a blow-up construction. 

\begin{defi}

Let $M$ be a compact 3-manifold and let~$(\varphi_X^t)_{t\in\RR}$ be a flow on $M$ generated by a smooth non-vanishing vector field $X$.
(Note that $X$ must be tangent to the boundary $\partial M$, which must therefore be toric.)
A \term{global cross section} for~$(M, (\varphi_X^t)_{t\in\RR})$ is a compact orientable surface with boundary~$S$ such that
\begin{itemize}
\item $S$ is embedded in~$M$ with $S\cap \partial M=\partial S$,
\item $S$ is positively transverse to~$X$,
\item every orbit of~$X$ intersects~$S$.
Note that the time to reach~$S$ is a continuous (and hence, bounded) function on~$M$.
\end{itemize}
\end{defi}

When such a global cross section exists, there is a well defined, bijective first-return map~$f$ on~$S$ and the first-return time~$\tau$ is bounded from above and below by compactness.
In this case~$M$ fibers over the circle with fiber~$S$, so that~$M$ equipped with the vector field $X$ is homeomorphic to~$S\times[0,1]/_{(p,1)\sim(f(p),0)}$ equipped with $\tau(p)\frac{\partial}{\partial z}$, where~$\frac \partial{\partial z}$ denotes the vector field tangent to the last coordinate and $f(p) = \varphi^{\tau(p)}(p)$ is the first-return map.
The dynamics of the flow~$(\varphi_X^t)_{t\in\RR}$ is then, up to the time-reparametrisation function~$\tau$, the dynamics of the map~$f$.
  
The following remark is folklore, see for example the discussion at the beginning of~\cite[Section 3]{thurston1986norm}.
It suggests that questions of existence of global cross sections are of algebraic nature.
  
\begin{prop}\label{P:isotopy}
For $(M, (\varphi^t)_{t\in\RR})$~a flow and~$S_1, S_2$ two global cross sections, there is an isotopy along orbits of~$(\varphi^t)_{t\in\RR}$ that sends $S_1$ on~$S_2$ if and only if $S_1$ and~$S_2$ represent the same class in~$\H_2(M, \partial M; \ZZ)$.
\end{prop}

\begin{proof}
The direct implication $\implies$ is obvious.
For the converse, let $\widehat M$ be the infinite cyclic cover of~$M$ associated to the class $[S_1] = [S_2] \in \H_2(M, \partial M; \ZZ)$ ($= H^1(M;\ZZ)$ by Lefshetz duality).
By construction, the surface~$S_1$ lifts to $\ZZ$ distinct parallel copies~$(S^{(n)}_1)_{n\in\ZZ}$.
The flow~$(\varphi^t)_{t\in\RR}$ lifts to a flow~$(\hat\varphi^t)_{t\in\RR}$ in~$\widehat M$.
Since~$S_1$ intersects all orbits of~$(\varphi^t)_{t\in\RR}$, every orbit of~$(\hat\varphi^t)_{t\in\RR}$ intersects each of the surfaces $(S^{(n)}_1)_{n\in\ZZ}$ one after the other.

Now, $S_2$ also lifts to $\ZZ$ parallel copies in~$\widehat M$ with the same property.
In particular every orbit of~$(\hat\varphi^t)_{t\in\RR}$ intersects exactly once each of the surfaces~$S^{(0)}_1$ and~$S^{(0)}_2$.
Hence for $p\in S^{(0)}_1$, we can define~$t_p$ to be the unique time so that~$\hat\varphi^{t_p}(p)\in S^{(0)}_2$.
The isotopy~$(f_s:p\mapsto \hat\varphi^{st_p}(p))_{s\in[0,1]}$ hence connects~$S^{(0)}_1$ to~~$S^{(0)}_2$ along orbits of~$(\hat\varphi^t)_{t\in\RR}$.
Projecting back to~$M$ yields the result.
\end{proof}

Note that if we are given a global cross section~$S$, it intersects all orbits positively.
So, taking homology classes, we see that the class~$[S] \in \H_2(M, \partial M; \ZZ)$ intersects positively all homology classes of periodic orbits of the flow.
One may wonder whether the above remark can be turned into a sufficient condition: when does a given homology class~$\sigma$ in~$\H_2(M, \partial M; \ZZ)$ contain a global section?

The answer has been given by Sol Schwartzman~\cite{schwartzman1957asymptotic} and Francis Fuller \cite{fuller1965surface}, and rephrased by Dennis Sullivan~\cite{sullivan1976cycles}.
The quicker way to express it requires to consider invariant measures as currents and to consider their homology classes:
given an $X$-invariant probability measure~$\mu$, the associated 1-current~$c_\mu$ is the linear functional on the space~$\Omega^1(M)$ of 1-forms defined by $c_\mu(\lambda) = \int_M \lambda(X(p)) d\mu(p)$.
Since~$\mu$ is invariant, $c_\mu$ is closed as a current, hence it induces a homology class $[c_\mu]$ in $\H_1(M; \RR)$.
The latter is called the \term{Schwartzman asymptotic cycle} associated to~$\mu$.
The set of all asymptotic cycles is denoted by~$\Schw{X}$.
It is a convex subset of~$\H_1(M; \RR)$ which contains the classes of the periodic orbits (consider the Dirac linear invariant measures carried by periodic orbits).
The following criterion is due to Schwartzman in the case $M$ has no boundary, and to Fried when~$\partial M$ is non-empty~\cite{schwartzman1957asymptotic, fried1982geometry}.
Here $\langle\cdot,\cdot\rangle_{(M, \partial M)}$ denotes the intersection pairing~$\H_2(M, \partial M; \RR)\times\H_1(M; \RR)\to\RR$.
Note that $H_2(M,\partial M;\ZZ) \subseteq H_2(M,\partial M;\ZZ) \otimes \RR = H_2(M,\partial M;\RR)$ by the universal coefficient theorem for homology.

\begin{theo}[Schwartzman--Fuller--Fried]
Let $M$ be a 3-manifold with toric boundary equipped with a non-vanishing vector field $X$ tangent to~$\partial M$.
A class~$\sigma$ in~$\H_2(M, \partial M; \ZZ)$ contains a global section for~$(M, X)$ if and only if for every asymptotic cycle~$c_\mu\in\Schw{X}$ one has $\langle \sigma, c_\mu\rangle_{(M, \partial M)} >0$.
\end{theo}

Now we turn to Birkhoff sections.
Recall from the introduction:

\begin{defi}
For $M$ a compact, orientable $3$-manifold with no boundary, $X$ a
non-vanishing vector field on~$M$ whose flow is denoted by $(\varphi_X^t)_{t\in\RR}$, an \term{embedded Birkhoff section} for $(M, (\varphi_X^t)_{t\in\RR})$ is a compact orientable surface~$S$\\\noindent
\begin{minipage}[t]{0.65\textwidth}
embedded in~$M$ such that
\begin{itemize}
  \item the interior of $S$ is positively transverse to~$X$,
  \item its boundary $\partial S$ is tangent to~$X$,
  \item we have $\varphi_X^{[0,T]}(S)=M$ for some~$T>0$.
\end{itemize}
\end{minipage}
\begin{minipage}[t]{0.35\textwidth}
  \begin{picture}(0,0)
  \put(2,-17)  {\includegraphics[width=40mm]{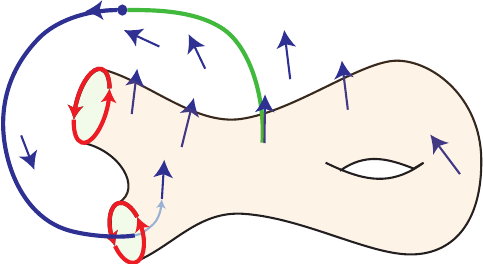}}
  \label{F:BirkhoffSection}
  \end{picture}
\end{minipage}
\end{defi}

\medskip

The second condition implies that the boundary of~$S$ is the union of finitely many periodic orbits of~$X$.
Note that one sometimes allows the boundary of~$S$ to be immersed instead of embedded, as in~\cite{colin2022generic}.
In such case we say that $S$ is an {\it immersed Birkhoff section}.

The first and second conditions in the definition of a Birkhoff section may look hard to realize at the same time, but actually it is not the case: in a flow box oriented so that the vector field is vertical, the general picture of an embedded Birkhoff section near its boundary is that of one helicoidal staircase. Since the interior of a Birkhoff section~$S$ is transverse to~$X$, it is cooriented by~$X$.\\\noindent
\begin{minipage}[t]{0.60\textwidth}
Since $M$ is oriented, this induces an orientation on~$S$, and in turn an orientation of~$\partial S$.
On the other hand, $\partial S$ is a collection of periodic orbits of~$X$, so it is oriented by~$X$.
For every component~$\beta$ of~$\partial S$, we can then define the multiplicity of $\beta$ as the algebraic number of times one sees $\beta_i$ in~$\partial S$.
Since we restrict our attention to embedded Birkhoff sections, this multiplicity is always $\pm1$.
We call a Birkhoff section~\term{positive} ({\it resp.} \term{negative}) if every boundary component has multiplicity $+1$ ({\it resp.} $-1$).
\end{minipage}
\begin{minipage}[t]{0.40\textwidth}
  \begin{picture}(0,0)(0,0)
  \put(5,-34){\includegraphics[width=45mm]{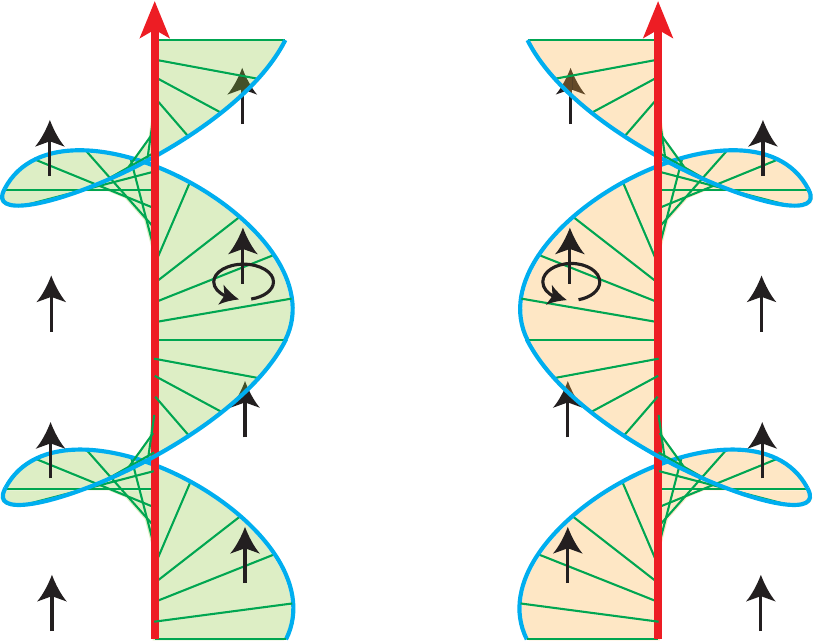}}
  \put(7,-37){\emph{negative}}
  \put(35,-37){\emph{positive}}
  \label{F:Boundary}
  \end{picture}
\end{minipage}

\smallskip

The connection with global cross sections comes from the following remark :
starting from a non-singular flow~$X$ on a compact 3-manifold~$M$ with no boundary, and given a finite collection~$\beta$ of periodic orbits of~$X$, one can consider the \term{normal blow-up} of~$M$ along~$\beta$, denoted by~$M_\beta$.
It is obtained from $M$ by removing the 1-submanifold~$\beta$ and replacing it by its unit normal bundle~$\nu^1_X(\beta)$.
In this construction, each component of~$\beta$ is replaced by a torus.
If $X$ is of class $C^1$, it extends to~$\nu^1_X(\beta)$ via its differential, so that $M_\beta$ is equipped with a continuous vector field~$X_\beta$.

Now if $S$ is a global cross section for~$(M_\beta, (\varphi_{X_\beta}^t)_{t\in\RR})$, one can change it by an isotopy in an arbitrarily small neighborhood of~$ \partial M_\beta$, so that every boundary component of~$\partial S$ is
\begin{itemize}
\item either a meridian circle of a boundary torus, that is, the normal bundle to a point~$p\in\beta$,
\item or a longitude of a boundary torus, that is, its projection in~$M$ is an immersion.
\end{itemize}
After such an isotopy, by blowing down the components of~$\partial S$ into orbits of~$X$, we obtain an immersed Birkhoff section for~$(M, (\varphi_X^t)_{t\in\RR})$ whose boundary is in~$\beta$.
Therefore global cross sections for $(M_\beta, (\varphi_{X_\beta}^t)_{t\in\RR})$ up to isotopy induce Birkhoff sections whose boundary is in~$\beta$ up to isotopy fixing the boundary.

Conversely, starting from a Birkhoff section~$S$, one can blow up its boundary and obtain a global cross section on the blown-up 3-manifold.

Therefore, provided one can understand the Schwartzman asymptotic cycles after blowing up a periodic orbit, one can adapt the Schwartzman--Fried Criterion to the existence of Birkhoff sections.
This was done by Fried and even precised by Hryniewicz~\cite[Thm N]{fried1982geometry}, \cite{hryniewicz2020note}, as we now explain. 
In our context of a vector field~$X$ on a 3-manifold $M$ with a specified finite set~$\beta$ of periodic orbits, every $X$-invariant measure can be split into two parts: one that is supported on~$M\setminus\beta$ and then descends to a $X_\beta$-invariant measure on~$M_\beta$, and one part that corresponds to a combination of Dirac linear $X$-invariant measures on the components of~$\beta$.
This second part has to be replaced on~$M_\beta$ by an $X_\beta$-invariant measure on~$\nu^1_X(\beta)$.
Since a flow on a 2-torus is in general not uniquely ergodic, the unit normal bundle~$\nu^1_X(\beta)$ admits several $X_\beta$-invariant measures.
However, a given class~$\sigma$ in~$\H_2(M, \beta; \ZZ)$ induces a class, also denoted by~$\sigma$, in~$\H_2(M_\beta, \partial M_\beta; \ZZ)$.
All asymptotic cycles associated to all $X_\beta$-invariant measures concentrated on~$\nu^1_X(\beta)$ have the same pairing with~$\sigma$, which corresponds to the rotation number of~$X_\beta|_{\nu^1_X(\beta)}$ with respect to the slope induced by~$\partial\sigma$.
We call this pairing the \term{self-linking} of~$\beta$ along~$X$ associated to the framing given by~$\sigma$, and denote it by~$\langle\partial\sigma, \beta^X\rangle_{\nu^1(\beta)}$.



\begin{theo}[Schwartzman--Fuller--Fried--Hryniewicz]\label{T:Criterion}
Given are a compact 3-manifold~$M$ with no boundary, a non-vanishing vector field $X$ on~$M$, and a finite collection~$\beta$ of periodic orbits of~$X$.
Then a class $\sigma$ in $\H_2(M, \beta; \ZZ)$ contains an embedded Birkhoff section for~$(M, (\varphi_X^t)_{t\in\RR})$ if and only if
\begin{itemize}
\item for every $X$-invariant measure~$\mu$ whose support does not intersect~$\beta$, the corresponding asymptotic cycle~$c_\mu\in\Schw{X}$ satisfies $\langle \sigma, c_\mu\rangle_{(M, \beta)} >0$,
\item for every component~$\beta_i$ of~$\beta$, the boundary of~$\partial \sigma$ travels plus or minus once along~$\beta_i$, and one has~$\langle\partial\sigma, \beta_i^X\rangle_{\nu^1(\beta_i)}>0$.
\end{itemize}
\end{theo}

\subsection{Anosov flows}\label{S:Anosov}
Geodesic flows on unit tangent bundles to hyperbolic surfaces are archetypes of transitive Anosov flows~\cite{hadamard1898surfaces,anosov1969geodesic}.
As such, their asymptotic cycles are easier to understand than those of general flows, as we now explain.

\noindent
\begin{minipage}[t]{0.67\textwidth}
\hspace{3mm}
Recall that a flow~$(\varphi_X^t)_{t\in\RR}$ generated by a vector field $X$ on a 3-manifold is of \term{Anosov type} if there are two transverse $\varphi_X$-invariant 2-foliations~$\Fs, \Fu$ on~$M$ that intersect along~$\RR.X$, where $X$ is the generator of the flow, such that $\Fs$ is transversally exponentially contracted by~$\varphi_X^t$ when $t\to+\infty$ and $\Fu$ is transversally exponentially contracted by~$\varphi_X^{t}$ when $t\to-\infty$.\footnotemark
\end{minipage}
\begin{minipage}[t]{0.33\textwidth}
  \begin{picture}(0,0)
  \put(1,-25){\includegraphics[width=40mm]{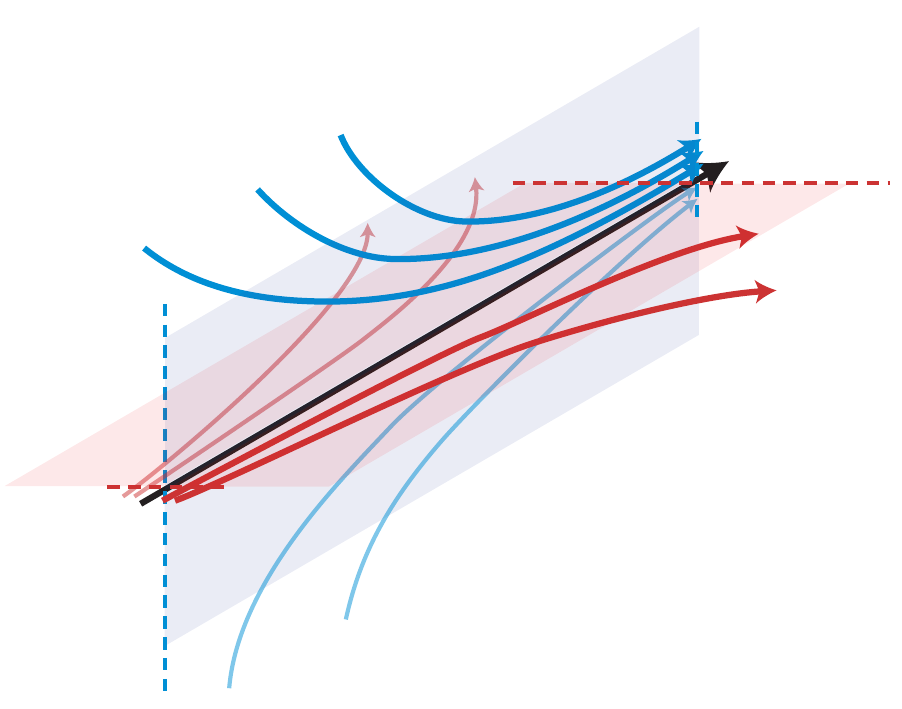}}
  \label{F:AnosovFlow}
  \end{picture}
\end{minipage}
The leaves of~$\Fs$ and~$\Fu$ are called \term{stable} and \term{unstable manifolds} respectively. 
\footnotetext{Actually this definition corresponds to \emph{topologically Anosov} flows, which is enough for us, as the results we use hold for topologically Anosov flows. Note that it was proven by Shannon that transitive topological Anosov flows are topologically equivalent to smooth Anosov flows~\cite{shannon2020dehn}, so that the topological results on transitive smooth Anosov flows can be used for topologically Anosov flows.}

Recall that two flows are \term{orbitally equivalent} if there is a homeomorphism sending the oriented orbits of the first flow onto the oriented orbits of the second one.
The geodesic flow on a hyperbolic surface is of Anosov type~\cite{anosov1969geodesic}.
In particular it is structurally stable, meaning that a small enough perturbation of the generating vector field yields an orbitally equivalent flow.
Together with the connectedness of the space of hyperbolic metrics, this implies that the geodesic flows associated to two different hyperbolic metrics are orbitally equivalent~\cite{gromov1976three}.
This means that, as long as only the topological properties of orbits are involved,
the geodesic flows of all possible hyperbolic metrics on a given surface are equivalent.

Blowing-up some periodic orbits of an Anosov flow does not yield an Anosov flow.
However it preserves the pseudo-Anosov character, so we rather work in this context.

Consider the unit disc~$\DD^2$ in~$\CC$.
For any integer $k \ge 3$ consider the singular 1-foliation~$\FF^1_k$ on~$\DD^2$ given by $d(\Re(z^{k/2}))=0$, and denote by $\FF^2_k$ the singular 2-foliation $\FF^1_k\times (0,1)$ on~$\DD^2\times(0,1)$.
The leaf~$\{0\}\times (0,1)$ is singular.
Also consider the half-unit disc~$\UU^2 = \DD^2\cap \{\Im(z)>0\}$.
Consider the singular 1-foliation $\FF^1_\partial$ on $\UU^2$ given by $d(\Re(s)\cdot\Im(z))=0$, and denote by $\FF^2_\partial$ the singular 2-foliation $\FF^2_\partial\times (0,1)$ on~$\UU^2\times(0,1)$.
The leaf~$\{0\}\times (0,1)$ is also singular.

Given a compact 3-manifold $M$ with toric boundary, a \term{foliation with circle-prongs} of $M$ is a 2-foliation with singularities~$\FF$ of~$M$ locally modelled on a standard 2-foliation or on some $\FF^2_k$ in the interior of~$M$, and on a standard 2-foliation tangent to~$\partial M$ or on~$\FF^2_\partial$ along~$\partial M$, see Figure~\ref{F:CircleProngs}.

\begin{figure}[htb]
  \begin{picture}(100,33)
  \put(0,2){\includegraphics[width=100mm]{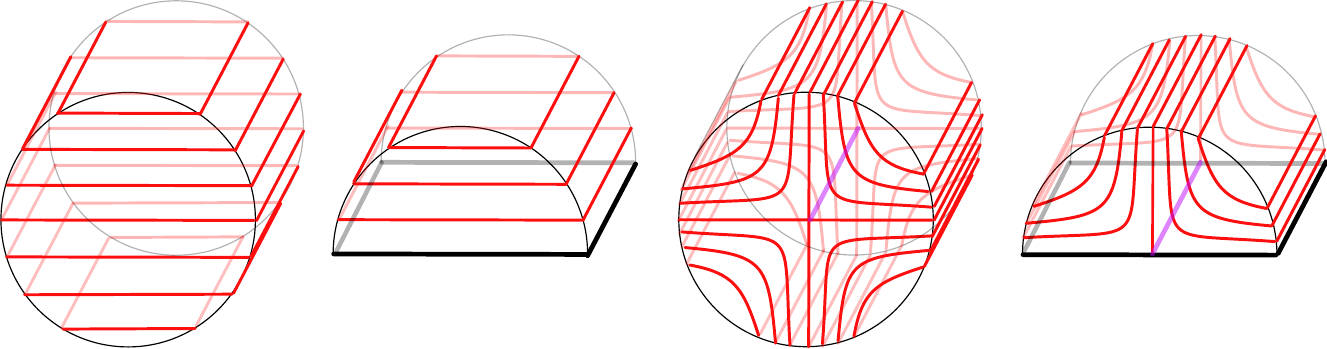}}
  \put(18,0){(a)}
  \put(43,0){(b)}
  \put(70,0){(c)}
  \put(95,0){(d)}
  \end{picture}
  \caption{\small The local picture of a standard 2-foliation in the interior of a 3-manifold (a) and on the boundary (b). The local picture of the foliation with circle-prong~$\FF^2_4$ (c) and the local picture of~$\FF^2_\partial$ (d).}
  \label{F:CircleProngs}
\end{figure}

A flow~$(\varphi^t)_{t\in\RR}$ on~$M$ is of \term{pseudo-Anosov type} if there are two $(\varphi^t)_{t\in\RR}$-invariant foliations with circle-prongs~$\Fs, \Fu$ on~$M$ that are transverse to each other and intersect along~$\RR.X$ (except along the singular curves which are common, and parallel to $X$), where $X$ is the generator of the flow, such that $\Fs$ is transversally exponentially contracted by~$\varphi^t$ when $t\to+\infty$ and $\Fu$ is transversally exponentially contracted by~$\varphi^{t}$ when $t\to-\infty$.\footnote{Pseudo-Anosov flows correspond to the \emph{expansive flows} of Brunella~\cite{brunella1992expansive}. Thanks to results of Inaba--Matsumoto and Paternain, both notions coincide, as explained in Brunella's thesis.}
Note that the pseudo-Anosov flows we consider in the sequel are obtained by blowing up periodic orbits of Anosov flows.
Hence the circle-prongs of the blown-up foliations are only of type~$\FF^2_\partial$; the types~$\FF^2_k$ with~$k \ge 3$ do not appear in our context.

Recall that a flow is \term{transitive} if it has a dense orbit.
Geodesic flows on hyperbolic surfaces are transitive.
Brunella showed that transitive pseudo-Anosov flows admit finite Markov partitions~\cite[Thm 2.1]{brunella1992expansive}.
Earlier Fried showed that the cone generated by the asymptotic cycles of a flow admitting a finite Markov partition is easy to describe~\cite[Thm H]{fried1982geometry} :

\begin{theo}[Fried]\label{T:Markov}
Given a compact 3-manifold $M$ with toric boundary and a non-vanishing vector field~$X$ on~$M$ tangent to~$\partial M$ that generates a flow~$(\varphi_X^t)_{t\in\RR}$ admitting a finite Markov partition, then there is a finite collection~$\{\beta_1, \dots, \beta_n\}$ of periodic orbits of~$(\varphi_X^t)_{t\in\RR}$ such that $\RR_+.\Schw{X} = \Conv(\{\RR_+[\beta_i]\}_{i=1, \dots, n})$.
\end{theo}

Combining the above statement with the existence criterion of Theorem~\ref{T:Criterion}, in the case of geodesic flows we obtain the following result.

\begin{coro}\label{T:CriterionBirkhoff}
Given a hyperbolic surface~$\Sigma$ and $\vec\beta$ a signed collection of periodic orbits of~$\fgeodt$ on~$\U\Sigma$, a class $\sigma$ in~$\H_2(\U\Sigma, \vec\beta; \ZZ)$ such that $\partial \sigma = \vec\beta$ contains a Birkhoff section for~$\fgeodt$ if, and only if,
\begin{itemize}
\item for every periodic orbit~$\vec\alpha$ of~$\fgeodt$ not in~$\vec\beta$, one has $\langle \sigma, [\vec\alpha]\rangle_{(\U\Sigma, \vec\beta)}>0$,
\item for every component~$\vec\beta_i$ of~$\vec\beta$, one has~$\langle\partial\sigma, \vec\beta_i^{X_{\mathrm{geod}}}\rangle_{\nu^1(\vec\beta_i)}>0$.
\end{itemize}
\end{coro}

Actually, Theorem~\ref{T:Markov} states that the infinite set of all periodic orbits in the first item could be replaced by a finite one, but determining this finite set for every signed collection~$\vec\beta$ does not look trivial to us.

\subsection{Classes of surfaces with given boundary}\label{S:Affine}

We come back to the setting of Theorem~\ref{T:Classification}: $\Sigma$ is a negatively curved surface, $\gamma$ is a finite collection of periodic geodesics and~$\vecgamma$ denotes the symmetric lift of~$\gamma$.
In order to apply Schwartzman--Fuller--Fried--Hryniewiez's criterion in the form of Corollary~\ref{T:CriterionBirkhoff} for finding Birkhoff cross sections bounded by~$-\vecgamma$, we need to work in the complement~$\U\Sigma\setminus~\vecgamma$ and in particular to determine the space~$\HSgammaZ$.
In this section we explain that the homology classes of 2-chains bounded by~$-\vecgamma$ form an affine space and we give a canonical origin to this space.

\begin{lemma}\label{L:Affine}
The homology classes of those 2-chains whose boundary is~$-\vecgamma$ form an affine space directed by~$\H_1(\Sigma; \ZZ)$.
\end{lemma}

\begin{proof}
First we consider the sequence~$0\to\H_2(\U\Sigma; \ZZ)\xrightarrow{i}\HSgammaZ \xrightarrow{\partial} \H_1(\vecgamma; \ZZ)$, where the first map is the inclusion map and the second is the boundary map.
We claim that it is exact\footnote{An erroneous version of this statement is in~\cite[Lemma 6]{fried1982geometry}, where it is claimed that the boundary map is surjective and admits a section. It is not true in general, unless $\U\Sigma$ is a homology sphere.}.
Indeed this is a part of the long exact sequence associated to the pair~$(\U\Sigma, \vecgamma)$, see~\cite[Thm 2.16]{hatcher2002algebraic}, plus the fact that~$\H_2(\vecgamma; \ZZ)$ is zero.

Now the homology classes of those 2-chains whose boundary is~$-\vecgamma$ correspond to the preimages under~$\partial$ of the point~$(-1, -1, \dots, -1)\in\H_1(\vecgamma; \ZZ)\simeq \ZZ^{2|\gamma|}$.
Indeed, given two 2-chains with the same boundary, their difference induces a class in~$\H_2(\U\Sigma; \ZZ)$.
Using the fact that~$\U\Sigma$ is a circle bundle with non-zero Euler class, we get~$\H_2(\U\Sigma; \ZZ)\simeq \H_1(\Sigma; \ZZ)$: a non-trivial class in $\H_2(\U\Sigma; \ZZ)$ can be represented by the set of the fibers over a cycle in~$\H_1(\Sigma; \ZZ)$.
\end{proof}

From Lemma~\ref{L:Affine} we deduce that if we are given an explicit 2-chain~$S\!_0$ bounded by~$-\vecgamma$, the classes of the other 2-chains bounded by~$-\vecgamma$ differ from~$[S\!_0]$ by a class in~$\H_1(\Sigma; \ZZ)$.
In our context, there is a natural choice of such an origin~$S_0$, for which the computation of the intersection numbers with asymptotic cycles of the geodesic flow will be easy.

We denote by~$S^\times_{\pm}$ the rational chain in~$\mathrm{C}_2(\U\Sigma, \vecgamma; \QQ)$ that is half the sum of all elementary rectangles~$R^{e, \eta}$ (see Figure~\ref{F:S0}) and by~$\sigma_\pm$ its homology class in~$\H_2(\U\Sigma, \vecgamma; \QQ)$,
\[S^\times_{\pm} := \frac 12 \sum_{e\in\gamma, \eta_e=\pm} R^{e, \eta_e},\qquad
\sigma_\pm := [S^\times_{\pm}].\]
In other words, we consider the set of all tangent vectors based at points of~$\gamma$.
Remember that every elementary rectangle is cooriented by the geodesic flow, hence oriented.
Therefore, $S^\times_{\pm}$ is also oriented.
Its boundary is then exactly~$-\vecgamma$ (thanks to the $\frac 1 2$ factor).
The 2-chain $S^\times_{\pm}$ is not a surface since the fibers of the double points of~$\gamma$ are singular.
As it is rational the class $\sigma_\pm$ might not be realized by a surface, but $2\sigma_\pm$ is always an integral class\footnote{Actually, $\sigma_\pm$ is realized by a surface if and only if~$[\gamma]_2$, the class of~$\gamma$ in $\H_1(\Sigma; \ZZ/2\ZZ)$, is $0$. In this case, the homology class of Birkhoff's coorientation~$\eta_{B}$ (Example~\ref{Ex:Birkhoff}) is $0$, and~$S^{BB}(\eta_B)$ lies in the class~$\sigma_\pm$.
Also the class $\sigma_\pm$ is equal to $\frac 1 2[\Snu+\Snnu]$ for every Eulerian~$\eta$.
Hence it is always realized as the mean of two surfaces without any assumption on~$[\gamma]_2$.}.

\begin{figure}
\centering
\includegraphics[width=.5\textwidth]{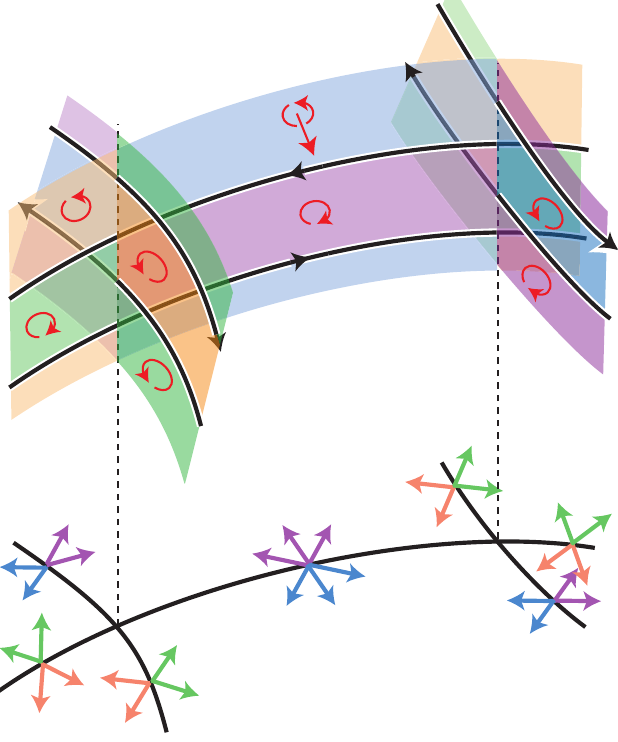}
\caption{\small The 2-chain $S^\times_{\pm}$ is half of the sum of all rectangles~$R^{e, \eta_e}$. It is cooriented by the geodesic flow, hence oriented (in red). Its boundary, taking orientations into account, is~$-\symgamma$. }
\label{F:S0}
\end{figure}

The class~$[S^\times_{\pm}]$ yields a canonical origin to the affine space of those 2-chains bounded by~$-\symgamma$, in the sense that it connects the intersection numbers in~$\U\Sigma_{\vecgamma}$ to intersection numbers of the base surface~$\Sigma$.

\begin{lemma}\label{L:Intersection}
For $\abas$ a collection of oriented periodic geodesics on~$\Sigma$, none of which is a component of~$\gamma$, denote by~$\vec\alpha$ its lift in~$\U\Sigma$.
Then the algebraic intersection~$\langle \sigma_\pm,\vec\alpha\rangle_{(\U\Sigma, \vecgamma)}$ is equal to~$+\frac12\Len_\gamma(\abas)$.
\end{lemma}

This lemma appears in a different form in~\cite{duke2017modular} where it is used to prove that the linking number of two  symmetric collections~${\overset\leftrightarrow{\gamma_1}}, {\overset\leftrightarrow{\gamma_2}}$ in~$\U\Sigma$ is equal to~$-\Len_{\gamma_1}(\gamma_2)$.

\begin{proof}
Since $S^\times_{\pm}$ is positively transverse to the geodesic flow, all intersection points of~$\vec\alpha$ with~$S^\times_{\pm}$ contribute positively to the algebraic intersection.
Since every rectangle has coefficient~$\frac12$  in $S^\times_{\pm}$, the contribution of every intersection point is~$+\frac 12$.
Finally $\vec\alpha$ intersects~$S^\times_{\pm}$ exactly in the fiber of the intersection points of~$\abas$ and~$\gamma$.
\end{proof}

The connection with intersection norms is now straightforward:

\begin{coro}\label{C:norm}
For $\abas$ a collection of oriented periodic geodesics on~$\Sigma$, none of which is a component of~$\gamma$, the intersection~$\langle \sigma_\pm,\vec\alpha\rangle_{(\U\Sigma, \vecgamma)}$ is at least equal to $\frac12\xx([\abas])$, with equality if and only if $\abas$ is an $\xx$-realizing collection of geodesics.
\end{coro}

\subsection{Proofs of Proposition~\ref{T:Brunella} and Theorem~\ref{T:Classification}}
\label{S:Proof}

Let us recall the context: $\Sigma$ is a hyperbolic surface and $\gamma$ a finite collection of periodic orbits on~$\Sigma$.
We denote by~$\vecgamma$ the symmetric lift of~$\gamma$ in~$\U\Sigma$ and by~$\U\Sigma_{\vecgamma}$ the 3-manifold obtained from~$\U\Sigma$ by blowing up the link~$\vecgamma$.
It has toric boundary, and it is equipped with the extension, also denoted by~$\fgeodt$, of the geodesic flow.
The latter is of pseudo-Anosov type (see Section~\ref{S:Anosov}).

Denote by~$\pi_*$ the canonical projection from $\H_2(\U\Sigma; \RR)$ to~$\H_1(\Sigma; \RR)$.
The next statement is the key property connecting Birkhoff sections and intersection norms.

\begin{lemma}\label{L:PositiveIntersection}
If $\gamma$ is a filling geodesic multicurve on~$\Sigma$, a class~$\sigma\in\HSgammaZ$ intersects positively ({\it resp.} non-negatively) every class~$[\vec\alpha]\in\H_1(\U\Sigma_{\vecgamma}; \ZZ)$ for $\abas$ an oriented periodic geodesic on~$\Sigma$ if and only if the class $\pi_*(\sigma-\sigma_\pm)\in\H_1(\Sigma; \ZZ)$ lies in the interior ({\it resp.} the closure) of~$\frac 12\BBx$.
\end{lemma}

\begin{proof}
For every oriented geodesic~$\abas$ on~$\Sigma$, by Lemma~\ref{L:Intersection}, we have
\begin{eqnarray*}
\big\langle \sigma, \vec\alpha\big\rangle_{(\U\Sigma,{\vecgamma})}
&=& \left\langle \sigma-\sigma_\pm, \vec\alpha\right\rangle_{(\U\Sigma, \vecgamma)} + \left\langle \sigma_\pm, \vec\alpha \right\rangle_{(\U\Sigma, \vecgamma)} \\
&=& \left\langle \sigma-\sigma_\pm, \vec\alpha\right\rangle_{(\U\Sigma, \vecgamma)} +\frac12\Len_\gamma(\abas)\\
&=& \left\langle \pi_*( \sigma-\sigma_\pm), \abas\right\rangle_{\Sigma}+\frac12\Len_\gamma(\abas).
\end{eqnarray*}
Hence $\left\langle \sigma, \vec\alpha\right\rangle_{(\U\Sigma, \vecgamma)}$ is positive if and only if $ -\left\langle \pi_*( \sigma-\sigma_\pm), \abas\right\rangle_{\Sigma}$ is smaller than~$\frac12\Len_\gamma(\abas)$.

Now the term~$ -\left\langle \pi( \sigma-\sigma_\pm), \abas\right\rangle_{\Sigma}$ depends only on the class~$[\abas]\in\H_1(\Sigma; \ZZ)$, while the other term~$+\frac12\Len_\gamma(\abas)$ is larger that~$\frac12\xxx{[\abas]}$, with equality if~$\abas$ is $\xx$-realizing (Corollary~\ref{C:norm}).

We then treat separately the cases $[\abas]\neq0$ and $[\abas]=0$ in~$\H_1(\Sigma; \ZZ)$.

Since there is an $\xx$-realizing geodesic in every non-zero homology class, the inequality \\$ -\left\langle \pi_*( \sigma-\sigma_\pm), \abas\right\rangle_{\Sigma} <+\frac12\Len_\gamma(\abas)$ is true for all non null-homologous geodesics~$\abas$ if and only if the inequality~$-\left\langle \pi_*( \sigma-\sigma_\pm), a\right\rangle_{\Sigma} < \frac 12\xxx{a}$ is true for every non-zero homology class.
In the same way, the inequality $ -\left\langle \pi_*( \sigma-\sigma_\pm), \abas\right\rangle_{\Sigma} \le +\frac12\Len_\gamma(\abas)$ is true for all non null-homologous geodesics~$\abas$ if and only if the inequality~$-\left\langle \pi_*( \sigma-\sigma_\pm), a\right\rangle_{\Sigma} \le \frac 12\xxx{a}$ is true for every non-zero homology class.

If $\abas$ is null-homologous, we have $\frac12\Len_\gamma(\abas)>0$ since the multicurve~$\gamma$ is filling, so that\\ $-\left\langle \pi_*( \sigma-\sigma_\pm), \abas\right\rangle_{\Sigma}=0$.

Summarizing the two previous paragraphs, we find that the class $\sigma$ intersects positively ({\it resp.} non-negatively) the class of every periodic orbit of the geodesic flow (in the complement of~$\symgamma$) if and only if for every class~$a\in\H_1(\Sigma; \ZZ)$ we have the inequality $-\left\langle \pi_*( \sigma-\sigma_\pm), a\right\rangle_{\Sigma} < \frac 12\xxx{a}$  ({\it resp.} $-\left\langle \pi_*( \sigma-\sigma_\pm), a\right\rangle_{\Sigma} \le \frac 12\xxx{a}$), which means exactly that the point~$-\pi_*( \sigma-\sigma_\pm)$ belongs to the interior ({\it resp.} the closure) of~$\frac12\BBx$.
Since the latter is symmetric about the origin, this amounts to $\pi_*( \sigma-\sigma_\pm)$ belonging to the interior ({\it resp.} the closure) of~$\frac12\BBx$.
\end{proof}

We can now assemble all blocks and prove our main results.

\begin{proof}[Proof of Theorem~\ref{T:Brunella}]
For~$\eta$ an Eulerian coorientation, we consider the Birkoff--Brunella surface $\Snu$ given by Definition~\ref{D:Snu}.
By Proposition~\ref{P:Transverse} its interior is transverse to the orbits of the geodesic flow in~$\U\Sigma$ while its boundary consists (with orientation) of~$-\symgamma$.
One checks that every elementary rectangle~$R^{e, \eta}$ contributes to~$-1$ to the Euler characteristics, hence~$\chi(\Snu)$ is $-|E(\gamma)|$.
Since~$\gamma$ is seen as a graph of degree~$4$, one has~$|E(\gamma)|=2|V(\gamma)|$, so that $\chi(\Snu)=-2|V(\gamma)|$.

Now if two Eulerian coorientantions~$\eta_1$ and $\eta_2$ are cohomologous, then the class~$[\Snuun-\Snudeux]\in\H_2(\U\Sigma; \ZZ)$ projects by~$\pi$ onto~$[\eta_1-\eta_2]=0$.
Since $\pi_*$ is actually an isomorphism we have $[\Snuun-\Snudeux]=0$, which in turn implies $[\Snuun]=[\Snudeux]$ in~$\HSgammaZ$.

Finally, if~$\Snuun$ and $\Snudeux$ are both Birkhoff sections of~$\fgeodt$ and are homologous, one can blow-up their boundary components (which are the same orbits), and Proposition~\ref{P:isotopy} claims that the flow actually realizes an isotopy between the blown-up surfaces. By blowing down, we obtain the desired isotopy away from the boundary.
\end{proof}

\begin{proof}[Proof of Theorem~\ref{T:Classification}]
Given a hyperbolic surface~$\Sigma$ and a geodesic multicurve~$\gamma$ that fills~$\Sigma$, Definition~\ref{D:Snu} yields a map that associates to every Eulerian coorientation $\eta$ of~$\gamma$ a surface~$\Snu$ bounded by~$-\vecgamma$ and whose interior is transverse to~$\fgeodt$.
Moreover, Proposition~\ref{T:Brunella} states that if two Eulerian coorientations~$\eta_1, \eta_2$ are cohomologous and the surfaces~$\Snuun, \Snudeux$ are Birkhoff sections for~$\fgeodt$, then they are actually isotopic along the flow.
Therefore the map~$S^{BB}$ projects to an injective map~$[S^{BB}]$ that takes a cohomology class of Eulerian coorientations to an isotopy class of surfaces transverse to~$\fgeodt$.

Lemma~\ref{L:Affine} claims that the homology classes of (rational) 2-chains bounded by~$-\vecgamma$ form an affine space directed by~$\H_1(\Sigma; \QQ)$.
The class~$\sigma_\pm$ defined in~\ref{S:Affine} gives a canonical origin to this space.
It is a half-integral class, and its double~$2\sigma_\pm$ is congruent to~$[\gamma]_2$ mod 2.
Therefore the set~$2\H_1(\Sigma; \ZZ)$ of the doubles of all integral classes corresponds to the sublattice of~$\H_1(\Sigma; \ZZ)$ of those points congruent to~$[\gamma]_2$ mod 2.

Theorem~\ref{T:Coor} states that all classes~$[\eta]$ for~$\eta$ Eulerian belong to the closure of~$\BBx$, and every integral point in~$\BBx$ that is congruent to~$[\gamma]_2$ mod~$2$ is realized by the class of an Eulerian coorientation.
This means that the domain of~$[S^{BB}]$ is exactly the integral classes in~$\BBx$ that are congruent to~$[\gamma]_2$ mod~$2$.

What remains to prove is that the restriction of~$[S^{BB}]$ to the interior of~$\BBx$ has its image in the realm of isotopy classes of Birkhoff sections, and that it is surjective.

By Schwartzman--Fuller--Fried-Hryniewicz's criterion in the form of Corollary~\ref{T:CriterionBirkhoff}, a class~$\sigma\in\H_2(\U\Sigma, \vecgamma; \ZZ)$ whose boundary is~$[-\vec\gamma]$ contains a Birkhoff cross section if and only if it pairs negatively with all classes~$\vec\alpha$ of periodic orbits of~$\fgeodt$, plus it links negatively with all boundary components (the $>0$ in~Corollary~\ref{T:CriterionBirkhoff} are all replaced by~$<0$ because of the signs of all boundary components).

By Lemma~\ref{L:PositiveIntersection} the first condition is equivalent to the difference~$\pi_*(\sigma-\sigma_\pm)$ lying inside~$\frac12\BBx$, or equivalently to~$2\pi_*(\sigma-\sigma_\pm)$ lying inside~$\BBx$.

Concerning the second condition in~\ref{T:CriterionBirkhoff}, one has to check that, if $\eta$ is an Eulerian coorientation such that~$[\eta]$ lies inside~$\BBx$, for every component~$\vec\gamma_i$ of~$\symgamma$, one has
$\langle\partial[S^{BB}(\eta)], \vec\gamma_i^{X_{\mathrm{geod}}}\rangle_{\nu^1(\vec\gamma_i)}>0$.
As explained just before Theorem~\ref{T:Criterion}, $\vec\gamma_i^{X_{\mathrm{geod}}}$ denotes any $X_{\mathrm{geod}}$-invariant measure in the boundary component of the blow-up of~$\vec\gamma_i$. 
One such invariant measure is carried by the trace of the stable manifold of~$\vec\gamma_i$, so that one only has to prove that 
$\partial\Snu\cap\nu^1(\vec\gamma_i)$ intersects the trace of the stable manifold of~$\vec\gamma_i$. 

Let us work by contrapositive and assume that~$\partial\Snu\cap\nu^1(\vec\gamma_i)$ does not intersect the stable manifold of~$\vec\gamma_i$.
Consider all double points that are met when travelling along the oriented curve~$\gbas$ on~$\Sigma$.
If at least one of them is alternating for~$\eta$ then one sees on Figure~\ref{F:BirSurface} that~$\Snu$ rotates from top to bottom (or bottom to top), so that it intersects the stable manifold of~$\vec\gamma_i$ in a neighbourhood of the considered double point.
Therefore~$\gbas$ has only vertices that are non-alternating for~$\eta$.
Moreover, looking at Figure~\ref{F:BruSurface}, one sees that at every vertex, the coorientation of the transverse component must be opposite to that of~$\gbas$.
Therefore the pairing~$\eta(\gbas)$ is~$0$, yielding~$[\eta]([\gbas])=0$, and so~$[\eta]$ does not lie in the interior of~$\BBx$.

The two previous paragraphs imply that the restriction of~$[S^{BB}]$ to the interior of~$\BBx$ has its image in the realm of isotopy classes of Birkhoff sections, and that it is surjective, thus concluding the proof.
\end{proof}

One may wonder what happens in Theorem~\ref{T:Classification} when~$\gamma$ is not filling.\footnote{In a previous version of this article, it was claimed that Theorem~\ref{T:Classification} also holds in this case. This is false, as was noted by T. Marty.}
In this case, there exists at least one geodesic~$\alpha$ not intersecting~$\gamma$.
The two oriented lifts of~$\alpha$ yield two periodic orbits~$\vec \alpha$ and~$\arev$ of~$\fgeodt$.
These two lifts are anti-isotopic in the complement of~$\symgamma$ : the isotopy obtained by rotating the tangent vectors by an angle from~$0$ to~$\pi$ transports $\vec \alpha$ to~$-\arev$.
This implies that a surface cannot be positively transverse to~$\vec \alpha$ and~$\arev$ simultaneously.
Therefore~$-\symgamma$ bounds no Birkhoff section.
However, the dual unit ball~$\BBx$ may or may not contain integral points in its interior, depending on~$\gamma$.
So there is no simple extension of Theorem~\ref{T:Classification} when $\gamma$ is not filling, except by saying that $-\symgamma$ cannot bound a Birkhoff section.

\section{Extension to orientable 2-orbifolds}

We explain here how the results extend to 2-dimensional orbifolds.
Actually Proposition~\ref{T:NORM}, Theorem~\ref{T:Coor} and Proposition~\ref{T:Brunella} extend directly.
The only point that is not straightforward is Theorem~\ref{T:Classification} which requires an additional argument.

\begin{defi}[chap.13 of \cite{thurston1980geometry}]\label{D:orbifold}
A \term{Riemannian orientable 2-dimensional orbifold}~$\OO$ is given by an orientable topological surface~$\Sigma_\OO$ together with an atlas $(U_\alpha, \varphi_\alpha)_{\alpha\in A}$ of charts of the form $\varphi_\alpha:U_\alpha\to D_\alpha/(\ZZ/k_\alpha\ZZ)$, with $D_\alpha$ a 2-dimensional Riemannian disc on which $\ZZ/k_\alpha\ZZ$ acts by rotations, and such that the chart transition maps $\varphi_\alpha\circ\varphi_\beta^{-1}$ are isometries.
\end{defi}

Actually the orbifolds to which our theorems extend are the hyperbolic ones.
Such a 2-orbifold is always \emph{good} in the sense of Thurston, namely it is a quotient of a hyperbolic surface by a finite automorphism group.

For our purpose we define the \term{first homology group} $\H_1(\OO; \RR)$ to be simply $\H_1(\Sigma_\OO; \RR)$.
Then the definition of intersection norms extends directly and Proposition~\ref{T:NORM} and Theorem~\ref{T:Coor} hold.

We now turn to Proposition~\ref{T:Brunella} and Theorem~\ref{T:Classification}.
First we have to define unit tangent bundles to orbifolds and geodesic flows.
If $D$ is a Riemannian disc on which $\ZZ/k\ZZ$ acts by rotation (with a fixed point), then $\ZZ/k\ZZ$ also acts on the unit tanget bundle~$\U D$.
The action on~$\U D$ is free, since the vectors tangent to the fixed point are rotated.
Hence the quotient $\U D/(\ZZ/k\ZZ)$ is a 3-manifold (actually it is a solid torus).

\begin{defi}\label{D:UOrbifold}
Given a Riemannian orientable 2-orbifold~$\OO=(\Sigma_\OO, (U_\alpha, \varphi_\alpha)_{\alpha\in A})$, its \dupeterm{unit tangent bundle} is the 3-manifold $\U\OO$ defined by the atlas $(\hat U_\alpha, \hat \varphi_\alpha)_{\alpha\in A}$, where $\hat U_\alpha =\U U_\alpha$ and $\hat \varphi_\alpha(x,v) = (\varphi_\alpha(x), d(\varphi_\alpha)_x(v))$.
It is equipped with a canonical projection $\pi:\U\OO\to\OO$.
If $\OO$ is of the form $\Sigma/\Gamma$ for some hyperbolic surface~$\Sigma$, then $\U\OO$ is simply the quotient~$(\U\Sigma)/\Gamma$.
The \dupeterm{geodesic flow} on~$\U\OO$ is defined as in the non-singular case by $\fgeod^t(\gamma(0), \dot\gamma(0))=(\gamma(t), \dot\gamma(t))$, where $\gamma$ is any geodesic with speed 1.
\end{defi}

With these definitions, the constructions of Section~\ref{S:Construction} (the BB-surface~$\Snu$ associated to an Eulerian coorientation) can be transposed and Lemmas~\ref{L:Snu}, \ref{L:Affine}, and \ref{L:Intersection} remain true.

Now, for $\OO$ a hyperbolic 2-orbifold, the unit tangent bundle $\U\OO$ is a 3-manifold, and we have $\H_2(\U\OO; \RR)\simeq\H_1(\OO; \RR)$.
Indeed closed curves in~$\Sigma_\OO$ lift by $\pi^{-1}$ to closed surfaces in~$\U\OO$.
The fact that the unit tangent to a conic disc~$D/(\ZZ/k\ZZ)$ is a torus whose core is the singular fiber implies that cohomologous curves lift to cohomologous surfaces, so that $\pi^{-1}$ induces a well defined map $\pi^{-1}_*:\H_1(\OO; \RR)\to \H_2(\U\OO; \RR)$.
The orbifold Euler characteristics of~$\OO$ is negative by hyperbolicity, so that the Euler number of~$\U\OO$ (as a Seifert fibered space) is also negative, hence the map $\pi^{-1}_*$ is an isomorphism.

Now Corollary~\ref{C:norm} holds, but Lemma~\ref{L:PositiveIntersection} needs to be adapted.
Firstly remark that if $\Sigma_\OO$ is a homology sphere, $\xx$ is the zero-function, so there is no possible interesting version of Lemma~\ref{L:PositiveIntersection} in this case.
Secondly, if $\Sigma_\OO$ is not a homology sphere, Lemma~\ref{L:PositiveIntersection} holds, but one argument needs to be developed, namely:

\begin{lemma}\label{L:OrbMin}
For $\OO$ a Riemannian orientable 2-orbifold and $\gamma$ a geodesic  on~$\OO$, for every non-zero homology class $a$ in $\H_2(\OO; \RR)$, there is an $\xx$-realizing geodesic in~$a$.
\end{lemma}

\begin{proof}
Let $\beta$ be an $\xx$-realizing curve such that $[\beta]=a$.
As in the case of a standard surface we want to strenghten $\beta$ to make it geodesic without changing the geometric intersection with~$\gamma$.
Far from the conic points, one can perform isotopies that shorten~$\beta$ with respect to the hyperbolic metric .
Since $\gamma$ is geodesic, these isotopies cannot increase the number of intersection (that is, no Reidemeister-II move is involved).

Around a conic point, one can work in a local conic chart.
This amounts to work on a standard disc where everything in invariant under a rotation.
Then one can also perform length-decreasing isotopies in an equivariant way, and this does not increase the number of intersection points with~$\gamma$.
\end{proof}

Proposition~\ref{T:Brunella} holds with no modification in the proof, and Theorem~\ref{T:Classification} has to be changed into Theorem~\ref{T:ClassificationBis} in order to treat the case of an orbifold whose underlying surface is a sphere.

\begin{proof}[Proof of Theorem~\ref{T:ClassificationBis}]
Suppose that~$\Sigma_\OO$ is a sphere.
Then $\U\Sigma_\OO$ is a rational homology sphere (in this case, $\H_1(\U\Sigma_\OO; \ZZ)$ is finite, but not reduced to the trivial group, unless $\Sigma_\OO$ is a sphere with three conic points of respective orders $2,3,$ and 7).
If $\gamma$ is filling, then the class~$\sigma_\pm$ intersects every asymptotic cycle, so it contains a Birkhoff section.
Since $\H_2(\U\Sigma_\OO; \ZZ)$ is trivial, all Birkhoff sections are homologous, hence isotopic relatively to their boundary.

If $\gamma$ is not filling, then there exists a geodesic~$\alpha$ not intersecting~$\gamma$ on~$\Sigma_\OO$.
Both its oriented lifts do not intersect~$S^\times_{\pm}$, hence there is an asymptotic cycle whose algebraic intersection with~$\sigma_\pm$ is zero.
Hence the class~$\sigma_\pm$ contains no Birkhoff section.
Since it is the unique class with boundary~$-\vecgamma$, there is no Birkhoff section bounded by~$-\vecgamma$ at all.

Finally if~$\Sigma_\OO$ is not a sphere and~$\gamma$ is filling, the norm~$\xx$ is non-degenerate, and the proof of Theorem~\ref{T:Classification} translates directly.
\end{proof}

\section{Questions}

\noindent{\bf On intersection norms.}
If~$\Sigma$ is a flat torus, then the minimal intersection is always realized by geodesics, which are unique in their homology class.
Hence if the collection~$\gamma$ is the union of $k$ geodesics $\gamma_1, \dots, \gamma_k$, then $i_\gamma(\alpha)=\sum_{i=1}^k i_{\gamma_i}(\alpha)$.
This implies that the dual ball~$B^*_{\gamma}$ coincides with the Minkowski sum $B^*_{\gamma_1}+\cdots+B^*_{\gamma_k}$.
Since the segment~$[-1, 1]\times\{0\}\subset\RR^2$ is the dual unit ball~$\BBx$ for~$\gamma$ the vertical circle on the torus, every segment containing~$0$ in the middle is the dual unit ball of some closed circle on the torus.
Therefore every convex polygon in~$\RR^2$ whose vertices are integral and congruent mod $2$ is of the form~$\BBx$ for some~$\gamma$.
This was already remarked by Thurston~\cite{thurston1986norm} and by Schrijver~\cite{schrijver1993graphs}.
In higher dimension the situation is probably more intricate.

\begin{ques}
Which polyhedra of~$\RR^{2g}$ with integer vertices can be realized as the dual unit ball~$\BBx$ for some~$\gamma$ in~$\Sigma_g$?
\end{ques}

A partial answer is given by Abdoul Karim Sane~\cite{sane2020intersection} who proves that some polyhedra in~$\RR^4$ cannot be dual unit ball of any intersection norm on a genus 2-surface.

Also, if~$\Sigma$ is a torus and $\gamma$ is a union of geodesics, then the above remarks imply that the number of self-intersection points of~$\gamma$ is exactly $1/4$ of the area of~$\BBx$ (check on Figure~\ref{F:Intro}).
Is there an analog statement in higher genus?

\begin{ques}
Which information concerning~$\gamma$ can be read on~$\BBx$? Is the number of self-intersection points of~$\gamma$ a certain function defined on~$\BBx$?
\end{ques}

This information is interesting since this number is exactly the opposite of the Euler characteristic of every Birkhoff cross section bounded by~$\vecgamma$.
Note that the number of self-intersection points is homogenous of degree~2, so we should look for degree~2 functions on polyhedra in~$\RR^{2g}$:
does it correspond to some symplectic capacity?

Motivated by our application we only defined the intersection norm for a collection of immersed curves, but one can directly extend it for an arbitrary embedded graph.
One can wonder which properties extend to this case and which information on the embedded graphs are encoded in this norm.
For example when the graph is Eulerian (\emph{i.e.,} all vertices have even degree) the connection with Eulerian coorientations remains.

\bigskip

\noindent{\bf On Birkhoff cross sections.}
Our constructions and our classification result deal only with Birkhoff cross sections bounded by a \emph{symmetric} collection of periodic orbits of the geodesic flow, that is, invariant under the involution~$(p,v)\mapsto(p,-v)$.
However the only restriction \emph{a priori} for being the boundary of a Birkhoff cross section is to be a boundary, that is, to be null-homologous.
Our results here say nothing about the classification, or even the existence, of Birkhoff cross sections with arbitrary null-homologous boundary.
In this case, the theory of Schwartzman--Fuller--Thurston--Fried and the remarks of Sections~\ref{S:SSF} and~\ref{S:Affine} still apply, so that these sections still correspond to the point inside a certain polytope in~$\H^1(\Sigma; \RR)$.
However we have no analog for the coorientations and the explicit constructions derived from them.

\begin{ques}
Is there a natural generalization of the polytope~$\BBx$ to non-symmetric finite collections~$\vec\gamma$ of closed orbits of the geodesic flow~$\fgeodt$, so that integer points in this polytope classify surfaces bounded by~$\vec\gamma$ and transverse to~$\fgeodt$?
\end{ques}

In the case of the flat torus, this question is answered in~\cite[Thm 3.12]{dehornoy2015geodesic} where a polygon~$P_{\vec\gamma}$ classifying transverse surfaces bounded by~$\vec\gamma$ is defined for \emph{every} null-homologous collection~$\vec\gamma$.

What would probably unlock the situation in the higher genus case would be to have, for every null-homologous collection~$\vec\gamma$, \emph{one} explicit surface bounded by~$\vec\gamma$ (not necessarily transverse), that is, an analog of $\sigma_\pm$ when $\vec\gamma$ is not symmetric.
Such an explicit point allows to compute its intersection with every other periodic orbit~$\vec\alpha$ of~$\fgeodt$.
These intersection numbers are all we need in order to describe explicitly the asymptotic directions of~$\fgeodt$ in~$\U\Sigma\setminus\vec\gamma$.
Generalising the constructions of~\cite{dehornoy2015genus} is a possibility here.

More generally, one can wonder whether there exists a generalization to all flows of the intersection norm~$\xx$ in the following sense:

\begin{ques}
For every 3-dimensional flow~$X$, is there an object that describes all isotopy classes of Birkhoff cross sections?
\end{ques}

A starting point would be to try with an Anosov flow that is not the geodesic flow, and see whether Gauss linking forms could play this role~\cite{ghys2009right}.

\section{Appendix: Thurston's theorem on integral seminorms}
\label{S:integral_seminorms}

Our goal in this section is to state and prove Thurston's theorem \cite[Thm.~2]{thurston1986norm} affirming that every integral seminorm $F$ defined on a lattice $L \simeq \ZZ^n$ is the pointwise maximum of a finite set $\Phi$ of linear functionals (i.e. homomorphisms $L \to \ZZ$).
In addition, we strengthen the conclusion of the theorem in the case that $F$ is equivalent modulo an integer $m \geq 1$ to a given homomorphism $\mu:L \to \ZZ_m$, affirming that in this case we can let $\Phi$ contain only homomorphisms that are equivalent to $\mu$ modulo $m$ as well.

Note that other proofs of Thurston's theorem have been given~\cite[Thm. 5]{fried1979fibrations}, \cite{delasalle2016norms}.
Here, we state the theorem in a way that involves only integer numbers (rather than reals, although the version for real numbers follows as a corollary).
The proof we give is similar to Thurston's original argument, but is written in greater detail and, like the statement of the theorem, relies only on the lattice and its dual, rather than extending the seminorm to a real vector space.
In fact, the proof yields an effective method for obtaining, for each integral seminorm $F$ and each vector $v$, a functional $\varphi \leq F$ that coincides with $F$ at $v$.

To facilitate the exposition we introduce the concept of a \term{narrow set} with respect to an integral seminorm $F$, which is any finite subset $X \subseteq V$ such that $F$ is linear 
on the semigroup spanned by~$X$.

\subsection{Definitions and statement of Thurston's theorem}

Recall that a \dupeterm{lattice} $L$ is a finitely generated free abelian group.
Its \dupeterm{dual lattice} $L^*$ is the group of homomorphisms $L \to \ZZ$.
Note that $L \simeq L^* \simeq \ZZ^n$ for some $n \in \NN$, called the \term{rank} of $L$.
The elements of $L$ and $L^*$ will be called \term{vectors} and \term{functionals}, respectively.
A \term{basis} of a lattice $L$ is an $n$-tuple $X=(x_i)_{i < n} \subseteq L$ such that every element of $L$ can be expressed by a unique integral combination of the elements $x_i$.

An \term{integral seminorm} on a lattice $L$ is a function $F:L \to \ZZ$ with the following two properties:
\begin{itemize}
\item \term{positive homogeneity}:
$F(\lambda \,v) = \lambda \, F(v)$ for all $v \in L$ and all scalars $\lambda \in \NN$,
\item \term{subadditivity}: $F(v+w) \leq F(v) + F(w)$ for all $v$,~$w \in L$.
\end{itemize}
(Note that we allow $F(-v) \neq F(v)$, and even $F(v) < 0$.)

Note first that every finite nonempty set of functionals $\Phi \subseteq L^*$ determines a integral seminorm $M_\Phi$ on $L$ given by
\begin{equation}\label{eq:seminorm_presentation}
M_\Phi(v) = \max_{\varphi \in \Phi} \varphi(v).
\end{equation}
Thurston's theorem asserts that in fact every integral seminorm is of this form.

\begin{theo}[Thurston's theorem on integral seminorms]
\label{thm:thurston}
Every integral seminorm $F$ on a lattice $L$ is of the form
\begin{equation}\label{eq:thurston}
F(v) = \max_{\varphi \in B^*_F} \varphi( v )
\end{equation}
where $B^*_F \subseteq L^*$ is the \term{dual unit ball} of $F$, that is, the set of all functionals $\varphi \in L^*$ satisfying $\varphi(v) \leq F(v)$ for all $v \in L$.
\end{theo}

\begin{remark}
The dual unit ball of any integral seminorm $F$ is finite, since the coefficients of a functional $\varphi \in B^*_F$ with respect to basis $E = (e_i)_{0 \leq i< n}$ are bounded by
$ \varphi_i = \varphi( e_i) \leq F( e_i)$
and 
$-\varphi_i = \varphi(-e_i) \leq F(-e_i)$.
\end{remark}


\begin{remark}\label{rmk:extremals_suffice}
For any finite set $\Phi$ of functionals on a lattice $L$ we have
\begin{equation}
\label{eq:extremals_suffice}
M_\Phi = M_{\ext(\Phi)},
\end{equation}
where $\ext(\Phi)$ denotes the set of \term{extremal points} of the set $\Phi$, that is, those points $\varphi \in \Phi$ that cannot be obtained as a (rational) convex combination of the elements of $\Phi\setminus\{\phi\}$.
Equation \eqref{eq:extremals_suffice} holds since every point of $\Phi$ is a convex combination of the extremal points of $\Phi$.
In particular, equation \eqref{eq:thurston} in Thurston's theorem is equivalent to
\begin{equation}\label{eq:thurston_extremals}
 F(v) = \max_{\varphi \in \ext(B^*_F)} \varphi( v ).
\end{equation}
\end{remark}

\begin{remark}\label{rmk:thurston_real}
The more commonly formulated version of Thurston's theorem, involving real numbers, is as follows.
A \term{seminorm} on a real vector space $V$ is a subadditive, positively homogeneous function $F:V \to \RR$, where positive homogeneity means that $F(\lambda \, v) = \lambda \, F(v)$ for all vectors $v \in \RR^n$ and scalars $\lambda \in \RR_{\geq 0}$.
Its \dupeterm{dual unit ball} $B^*_F$ is the set of (real-valued) functionals $\varphi \in V^*$ that satisfy $\varphi \leq F$ pointwise.
In this setting, Thurston's theorem asserts that any real seminorm $F$ on a vector space $V \simeq \RR^n$ taking integer values on some rank-$n$ lattice $L \subseteq V$ is of the form
\[ F(v) = \max_{\varphi\in \Phi} \varphi(v) \]
where $\Phi$ is the set of linear functionals $\varphi \in B^*_F$ that take integer values on $L$.
This version of Thurston's theorem follows readily from Theorem~\ref{thm:thurston}.
\end{remark}

\subsection{Proof of Thurston's theorem}

The proof is based on a method for verifying that a functional $\varphi$ on a lattice $L$ is in the dual unit ball of an integral seminorm $F$ after evaluating both functions at finitely many vectors.
To describe this method, we introduce the concept of \emph{narrow sets}.

To define this concept, we first recall some additional standard terminology.
Let~$L$ be a rank-$n$ lattice.
A \term{sublattice} of $L$ is a subgroup of $L$ (and is itself a lattice of rank $\leq n$ by the Smith normal form theorem),
and a \term{semigroup} in $L$ is any subset $S \subseteq L$ that is closed with respect to finite sums (including the empty sum).
Any set $X \subseteq L$ spans a sublattice $L_X$ and a semigroup $S_X$ consisting, respectively, of all integral or positive (i.e. non-negative) integral combinations of elements of $X$.

Now we can define narrow sets.
Note that, in essence, what we are trying to show in Thurston's theorem is that every integral seminorm is a piecewise-linear function.

\begin{defi}
A subset $X$ of a lattice $L$ is \term{narrow} with respect to a seminorm $F$ defined on $L$, or $F$-narrow, if on the semigroup $S_X$ spanned by $X$ the function $F$ is linear, that is, it coincides with some functional $\varphi \in L_X^*$.
\end{defi}

Note that there is at most one functional $\varphi \in L_X^*$ that coincides with $F$ on $X$, and in fact, exactly one if $X$ is linearly independent.
To determine whether $F$ coincides with $\varphi$ on the whole semigroup $S_X$ we have the following criterion.

\begin{prop}[Interior ray test]
\label{thm:interior_ray}
Consider a seminorm $F$ on a lattice $L$,
a vector tuple $X=(x_i)_{i\in k} \subseteq L$,
and a functional $\varphi \in L_X^*$ that is greater than or equal to $F$ at the vectors $x_i$ (and thus, at all vectors of the semigroup $S_X$).
Take an integer combination $c^+=\sum_i\alpha_ix_i$ with strictly positive coefficients $\alpha_i >0$, so that
\begin{equation}
\label{eq:interior_vector_ineq}
F( c^+ ) = F \left( \sum_i \alpha_i x_i \right)
\leq \sum_i \alpha_i F( x_i )
= \varphi( c^+ ).
\end{equation}
Then the following are equivalent:
\begin{enumerate}[(a)]
\item\label{it:at_vector} $F (c^+) \geq \varphi(c^+)$,
\item\label{it:on_space} $F \geq \varphi$ on the lattice $L_X$,
\item\label{it:equality} $F = \varphi$ on the semigroup $S_X$ (and hence $X$ is $F$-narrow).
\end{enumerate}
\end{prop}


The name of this result stems from the fact that the ray spanned by $c^+$ lies in the interior of the cone of positive combinations of the vectors $x_i$ (in the rational vector space $L_X \otimes_\ZZ \QQ$).
Proposition~\ref{thm:interior_ray} ensures that the function $F$ coincides with the linear function $\varphi$ on the whole cone if it coincides along this single ray.

\begin{proof}
We need just show that \ref{it:at_vector} implies \ref{it:on_space}, since the other forward implications are evident.
Suppose, then, that \ref{it:at_vector} holds, and take any vector $v \in L_X$.
We have to show that $F(v) \geq \varphi(v)$, and we do so as follows.

Recall the picture of the interior ray described above.
Since the ray spanned by~$c^+$ is in the interior of the cone spanned by $X$, it follows that there exists a vector $c \in S_X$ such that the ray spanned by~$c^+$ lies between those spanned by $c$ and by $v$.
More precisely, the sum $v+c$ is a positive multiple of $c^+$.

\begin{claim}
There exists a vector $c \in S_X$ and a number $\lambda \in \NN$ such that $v + c = \lambda c^+$.
\end{claim}

\begin{proof}[Proof of claim]
Recall that $c^+ = \sum_i \alpha_i x_i$ for some strictly positive integers $\alpha_i>0$, and since $v \in L_X$, we can also write $v = \sum_i \beta_i x_i$ using integers $\beta_i$.
Take a number $\lambda \in \NN$ such that $\lambda \alpha_i \geq \beta_i$ for all $i$.
Then we have $\lambda c^+ = v + c$ where
$c = \sum_i (\lambda \alpha_i - \beta_i) x_i$ is a vector of $S_X$ since $\lambda \alpha_i - \beta_i \geq 0$ for all $i$.
\end{proof}

Since~$F$ is subadditive and~$\varphi$ is additive, from the equation~$\lambda c^+ = c + v$ we infer that if the inequality~$F \leq \varphi$ holds at $c$ (which it does since $c \in S_X$) and also holds strictly at~$v$ (let us assume this, for a contradiction), then it also holds strictly at the vector~$\lambda c^+$.
However, we know that~$F(\lambda c^+) \geq \varphi(\lambda c^+)$ by the hypothesis~\ref{it:at_vector}. 
Therefore the inequality~$F \leq \varphi$ cannot hold strictly at~$v$,
which means that~$F(v) \geq \varphi(v)$, as we had to show.
\end{proof}


Now we are ready to prove Thurston's theorem.
Let $F$ be an integral seminorm on a lattice $L$, and take a vector $v \in L$.
We have to show that there exists a functional $\varphi\in B^*_F$ such that $\varphi(v) = F(v)$.
And for this, we may assume that $v$ is a \term{primitive element} of $L$ (that is, that it cannot be written as a multiple $v = \lambda w$ of an element $w \in L$ by an integer $\lambda>1$), since every element of a lattice is a positive multiple of some primitive element (again, by the Smith normal form theorem).

To prove that $F(v) = \varphi(v)$ for some $\varphi\in B^*_F$, it suffices to show that the vector $v$ is contained in the semigroup $S_X$ generated by some narrow basis $X$ of $L$, because this means that $F$ coincides with some functional $\varphi \in L^*$ on $S_X$, and this functional is in the dual unit ball $B^*_F$ by Proposition~\ref{thm:interior_ray}.
Therefore, to finish the proof of Thurston's theorem, it is enough to establish the following result.

\begin{prop}\label{thm:narrow_basis} If $F$ is an integral seminorm on a lattice $L$ then every primitive vector $v \in L$ is the first element of some $F$-narrow basis.
\end{prop}

The proof is constructive: if the integral seminorm $F$ is given as an oracle (or ``black box'') that outputs the value $F(v)$ for any given input vector $v \in L$, we will show how to obtain, after invoking this oracle finitely many times, an $F$-narrow basis $X$ containing $v$ and the corresponding functional $\varphi \in L^*$ that coincides with~$F$ on $S_x$, and therefore satisfies $\varphi(v) = F(v)$, and is in the dual unit ball $B^*_F$.

\begin{proof}[Proof of Theorem~\ref{thm:thurston}]
Let $X = (x_i)_{0 \leq i < n}$ be a basis of the lattice $L$ such that $x_0 = v$.
(Every primitive integral vector is part of a basis, which can be obtained by putting in Smith normal form the one-column matrix of coordinates of $v$ with respect to an initial arbitrary basis of $L$.)

In general the basis $X$ is not narrow, but we can modify it to make it narrow as follows.
We proceed by induction on the dimension.
Suppose that for some $k < n$, the $k$-tuple $X_k = (x_i)_{0 \leq i < k}$ is known to be narrow.
We may test whether $X_{k+1}$ is narrow by evaluating $F$ on the vector $x_k' := x_k + w$, where $w = \sum_{0 \leq i < k} x_i$.
Note that
\begin{equation}
\label{eq:increment_inequality}
F(x_k') \leq F(x_k) + F(w).
\end{equation}

\begin{claim}[increment test]
\label{thm:increment_test}
$X_{k+1}$ is narrow if and only if equality holds in \eqref{eq:increment_inequality}.
\end{claim}

\begin{proof}[Proof of claim]
Let $\varphi$ be the unique functional on the lattice $L_{X_{k+1}}$ that coincides with $F$ on $X_{k+1}$.
The vector $x_k'$ is the sum of all the vectors of $X_{k+1}$, thus, by Proposition~\ref{thm:interior_ray}, $X_{k+1}$ is narrow if and only if the inequality
$F(x_k') \leq \varphi(x_k')$ is an equality.
However, this inequality is equivalent to \eqref{eq:increment_inequality} since $\varphi(x_k') = \varphi(x_k) + \varphi(w) = F(x_k) + F(w)$.
(Here we used the equation $\varphi(w) = F(w)$, which holds because $w$ is a combination of the tuple $X_k$, that is assumed narrow.)
\end{proof}

If the increment test is not passed, we replace the vector $x_k$ in $X$ by the vector~$x_k'$, obtaining a new basis $X'$, and we redo the test.
(Note that this replacement is an elementary operation on $X$, therefore $X'$ is another basis of $L$.)
Since we may need to repeat this replacement many times, we denote $x_k^{(t)} = x_k + t w$, for $t \in \NN$ and we denote $X^{(t)}$ the basis obtained from $X$ by replacing $x_k$ by $x_k^{(t)}$.

\begin{claim}
For a large enough $t \in \NN$, the tuple $X^{(t)}_{k+1}$ is narrow.
\end{claim}

\begin{proof}[Proof of claim]
By the increment test described in Claim~\ref{thm:increment_test} above, it suffices to show that the inequality
\[ F(x_k^{(t+1)}) \leq F(x_k^{(t)}) + F(w) \]
is an equality for large enough $t$.
To prove this we consider the function
\[ f(t) = F(x_k^{(t)}) = F(x_k + t \, w). \]
Its discrete derivative $f'(t) := f(t+1)-f(t)$ is integer-valued, increasing (since~$f$ is convex), and bounded above by $F(w)$.
Therefore $f'$ eventually stabilizes at a constant value.
In fact, it stabilizes at the value $F(w)$.
We see this by comparing $f$ with the known function $g(t) = F(t w) = t F(w)$, whose difference with $f$ is bounded by the inequality
\[ g(t)
= F(x_k + t \, w - x_k )
\leq F( x_k + t \, w) + F( -x_k )
= f(t) + F( -x_k ). \qedhere \]
\end{proof}

By this process we find a narrow basis $X$ containing $v$ as its first element.
By Proposition~\ref{thm:interior_ray}, it follows that the unique functional $\varphi \in L^*$ that coincides with $F$ on $X$ is in the dual unit ball $B^*_F$ and satisfies $\varphi(v) = F(v)$, as we had to show.
\end{proof}

\subsection{Thurston's theorem for seminorms of a given class modulo $m$}

For an integer $m \geq 1$, denote $\ZZ_m$ the group of integers modulo $m$, and let $\pi_m:\ZZ \to \ZZ_m$ be the quotient map.
An integer-valued function $F$ on a lattice $L$ is \term{congruent} to a group homomorphism $\mu: L \to \ZZ_m$ if the function
$F_{\mmod m}:= \pi_m \circ F $
is equal to $\mu$.

Our goal now is to prove the following extension of Thurston's theorem.

\begin{theo}\label{thm:thurston_congruent}
Every integral seminorm $F$ on a lattice $L$ that is congruent modulo a certain integer $m \geq 1$ to a given homorphism $\mu: L \to \ZZ_m$ is of the form
\[ F(v) =
\max_{\substack{\varphi \in B^*_F\\\varphi_{\mmod m} = \mu}} \varphi( v ). \]
\end{theo}

To prove this result we use the following lemma.

\begin{lemma}\label{thm:extremal_dual_basis}
Let $F$ be an integral seminorm on a lattice $L$, and let $\varphi$ be an extremal functional of the dual unit ball $B^*_F$.
Then there exists a basis $X$ of $L$ such that $F$ coincides with $\varphi$ on the semigroup $S_X$.
\end{lemma}

\begin{proof}
Let $(\psi_i)_i$ be the functionals of the dual unit ball $B^*_F$ excluding $\varphi$.

We claim first that there exists some primitive vector $v\in L$ such that
$\psi_i(v) < \varphi(v)$ for all $i$.
Indeed, extremality of the functional $\varphi$ implies that it cannot be written as a rational convex combination of the functionals $\psi_i$.
By the Farkas lemma, it follows that there is a vector $v \in L$ (which can be taken primitive) such that $\varphi(v) > \psi_i(v)$ for all $i$.

Fixed the vector $v \in L$, we apply Proposition~\ref{thm:narrow_basis}, which ensures that the lattice~$L$ admits an $F$-narrow basis $X$ containing the vector $v$.
Since $X$ is $F$-narrow, the function $F$ coincides with some functional $\wt\varphi \in L^*$ on the semigroup $S_X$
(and in particular, at the vector $v$).
Moreover, this functional $\wt\varphi$ is in the dual unit ball $B^*_F$ by Proposition~\ref{thm:interior_ray}.
We conclude that $\wt\varphi = \varphi$ because $\varphi$ is strictly greater than all the other functionals $\psi_i \in B^*_F$ at the vector $v$.
\end{proof}

To finish, let us prove Theorem~\ref{thm:thurston_congruent}.

\begin{proof}[Proof of Theorem~\ref{thm:thurston_congruent}]
By Remark~\ref{rmk:extremals_suffice}, it suffices to show that every extremal point of the dual unit ball $B^*_F$ is congruent to $\mu$ modulo $m$.
Let $\varphi$ be an extremal functional of $B^*_F$.
By Lemma~\ref{thm:extremal_dual_basis}, there exists a basis $X$ of $L$ such that $F = \varphi$ on~$S_X$.
Suppose, for a contradiction, that $\varphi_{\mmod m} \neq \mu$.
This means that there is a vector $ v \in L$ such that
$ \varphi_{\mmod m}(v) \neq \mu(v) $.
Thus for each vector $v'$ of the set $ v + m L$ we have
\[ \varphi(v') \equiv \varphi(v) \not\equiv \mu(v) \equiv \mu(v') \mod m \]
since both $ \mu $ and $ \varphi_{\mmod m} $ vanish on the lattice $ m L$.
Take a vector $v' \in ( v + m L )$ whose coordinates with respect to the basis $X$ are positive, so that $v' \in S_X$, and hence we have
\[ F(v') = \varphi(v') \not\equiv \mu(v') \mod m, \]
contradicting the hypothesis $F_{\mmod m} = \mu$.
\end{proof}

\printbibliography

\bigskip
\noindent
{\footnotesize%
  \textsc{Marcos Cossarini}\par {Univ Paris-Est Créteil, CNRS, LAMA, Créteil, France} \par
  \url{marcos.cossarini@u-pec.fr} \par
}

\bigskip
\noindent
{\footnotesize%
  \textsc{Pierre Dehornoy}\par {Aix Marseille Univ, CNRS, I2M, Marseille, France} \par
  \url{pierre.dehornoy@univ-amu.fr} \par
  \url{https://www.i2m.univ-amu.fr/perso/pierre.dehornoy/}\par
}

\end{document}